\documentclass[11pt]{amsart}
\usepackage{graphicx, hyperref, mathtools, amsfonts, amssymb, amsthm}
\usepackage[margin=0.75in]{geometry}
\usepackage{float}

\usepackage{tikz-cd}

\usepackage{subcaption}
\usepackage{float}
\usepackage[backend=biber, sorting=nyt]{biblatex}
\newtheorem*{remark}{Remark}
\newtheorem{prop}{Proposition}[section]
\newtheorem{theorem}{Theorem}[section]
\newtheorem{corollary}[theorem]{Corollary}
\newtheorem{lemma}[theorem]{Lemma}
\theoremstyle{definition}
\newtheorem{definition}{Definition}[section]
\newtheorem{example}{Example}[section]

\newcommand{\tsi}[0]{\textup{TSI}}
\newcommand{\var}[0]{\textup{Var}}

\title{The Topological Stability Index: A Variance-Based Measure for Persistence Barcodes}

\author{Joris Kirchner}
\author{Ioannis Diamantis}

\address{Department of Data Analytics and Digitalisation, 
Maastricht University, School of Business and Economics, 
P.O. Box 616, 6200 MD Maastricht, The Netherlands}

\email{joris.kirchner@maastrichtuniversity.nl}
\email{i.diamantis@maastrichtuniversity.nl}

\begin{document}

\begin{abstract}
We introduce the \emph{Topological Stability Index} (TSI), a variance-based scalar measure for persistence barcodes that quantifies the dispersion of persistence lifetimes. Unlike persistent entropy, which depends only on normalized weights, the TSI captures absolute variability and is sensitive to heterogeneous feature scales. We establish fundamental properties of the TSI, including its scaling behavior, invariance under lifetime translation and explicit update formulas under insertion and deletion of bars. We also consider a complementary first-moment-type quantity, the \emph{Topological Signal Index} (TSigI), which captures the typical scale of persistence lifetimes and provides additional interpretability alongside the TSI. We further introduce a normalized version, $cv\tsi$, which is scale invariant and admits an explicit algebraic relation to the Rényi entropy of order two. In particular, $cv\tsi$ is an affine function of the collision probability $\sum_i p_i^2$, and therefore a monotone reparametrization of the Rényi entropy, providing a direct link between variance-based and entropy-based summaries in topological data analysis. Numerical experiments on synthetic data and stochastic time series demonstrate that the TSI captures structural variability complementary to entropy: it is relatively insensitive to deterministic trends, while responding strongly to stochastic fluctuations and variations in persistence magnitude.
\end{abstract}

\begingroup
\renewcommand\thefootnote{}
\renewcommand\footnoterule{}
\endgroup

\maketitle

\section{Introduction}

Topological Data Analysis (TDA) has emerged as a powerful framework for extracting geometric and structural information from complex data sets \cite{zomorodian_computing_2005}. At its core, TDA represents data (such as point clouds, time series, or networks) through a filtration, i.e., a family of nested topological spaces, and studies the evolution of topological features across scales using persistent homology. This approach is often motivated by the assumption that data exhibit an underlying geometric or topological structure, whose features persist across scales and encode meaningful information about the system under study.

A central object in TDA is the persistence diagram (or equivalently, the barcode), which records the birth and death times of topological features across the filtration. While persistence diagrams provide a rich and stable representation of topological structure, their integration with statistical methods remains a fundamental challenge \cite{cohen-steiner_stability_2007}. The space of persistence diagrams lacks a simple linear or convex structure, making it difficult to directly apply classical statistical tools \cite{Mileyko_2011}.

Several approaches have been developed to address this challenge. One direction focuses on performing statistical inference directly on persistence diagrams, including the construction of confidence regions and hypothesis testing procedures \cite{fasy_confidence_2014, robinson_hypothesis_2017}. More broadly, probabilistic and asymptotic frameworks have been proposed to study distributions of persistence diagrams \cite{mileyko_probability_2011, hiraoka_limit_2018}. While conceptually natural, these methods require substantial adaptations due to the geometric complexity of the diagram space.

A complementary line of work aims to improve robustness to noise and outliers by modifying the underlying filtration, for example through Distance-to-Measure (DTM) or kernel-based constructions \cite{chazal_robust_2018, anai_dtm-based_2020}. These approaches enhance stability properties of persistence diagrams, but typically introduce additional parameters and increased computational complexity.

A widely adopted alternative is to map persistence diagrams to different representations, enabling the use of standard statistical techniques \cite{ali_survey_2023}. Examples include functional summaries such as persistence landscapes \cite{bubenik_statistical_2015} and persistence curves \cite{chung_persistence_2022}, as well as scalar summaries. Among the latter, \emph{persistent entropy} \cite{rucco_characterisation_2016} has gained particular popularity. By applying Shannon entropy to the normalized distribution of persistence lifetimes, it provides a scale-invariant measure of the relative distribution of topological features and has been shown to be stable under perturbations \cite{atienza_persistent_2019, atienza_stability_2020}.

Despite these advances, an important aspect of persistence information remains less explored: the \emph{absolute dispersion} of persistence lifetimes. Entropy-based summaries depend only on normalized weights and therefore do not capture differences in scale or variability across barcodes. In many applications, however, such absolute differences are meaningful, as they reflect heterogeneity in persistence lifetimes and variations in structural complexity.

In this paper, we introduce the \emph{Topological Stability Index} (TSI), a variance-based scalar measure for persistence barcodes that quantifies the dispersion of persistence lifetimes. Complementing this, we consider a first-moment-type quantity, the \emph{Topological Signal Index} (TSigI), which captures the typical scale of persistence lifetimes. Together, these quantities provide a simple two-dimensional summary of persistence barcodes, encoding both magnitude and variability. The TSI was originally introduced in \cite{diamantis_shape_2025} as a conceptual tool; here, we develop it into a rigorous mathematical object and systematically investigate its properties. In particular, we establish its scaling behavior, invariance under lifetime translation, and explicit update formulas under insertion and deletion of bars.

To bridge the gap between variance-based and entropy-based summaries, we further introduce a normalized version of the TSI, denoted $cv\tsi$. We show that this quantity is scale invariant and admits an explicit algebraic relation to the Rényi entropy of order two. In particular, $cv\tsi$ is an affine function of the collision probability $\sum_i p_i^2$, and therefore a monotone reparametrization of the Rényi entropy. This connection provides a unified perspective on scalar summaries in TDA, linking dispersion-based and information-theoretic approaches.

Finally, we illustrate the behavior of the TSI and its normalized counterpart through numerical experiments on synthetic geometric data and stochastic time series. These experiments demonstrate that the TSI captures structural variability complementary to entropy: it is relatively insensitive to deterministic trends, while responding strongly to stochastic fluctuations and variations in persistence magnitude.\footnote{All experiments are fully reproducible; code is available at \url{https://github.com/siroj99/TSI/}.}

\bigbreak 

The rest of the paper is organized as follows. In Section~\ref{Sec:The Topological Stability Index}, we introduce the Topological Stability Index and establish its main properties. Section~\ref{Sec:Relation to Persistence Entropy} develops the normalized version and its connection to Rényi entropy. In Section~\ref{Sec:Numerical Experiments and Applications}, we present numerical experiments and applications, illustrating the behavior of the proposed summaries on geometric and stochastic data. Finally, Section~\ref{Sec:Discussion and Future Work} discusses implications, limitations, and directions for future research.

\section{The Topological Stability Index}\label{Sec:The Topological Stability Index}

Persistence barcodes provide a rich representation of topological features across scales, but their direct comparison and interpretation often require suitable summary statistics. Several scalar descriptors have been proposed in the literature, many of which are based on normalized quantities and therefore primarily capture relative information about the distribution of persistence lifetimes.

In this section, we introduce the \emph{Topological Stability Index} (TSI), a simple yet expressive statistic defined as the variance of persistence lifetimes. Unlike entropy-based summaries, the TSI retains information about absolute scale and dispersion, making it sensitive to structural heterogeneity in the barcode. We develop its basic properties, analyze its behavior under fundamental transformations, and study its response to perturbations such as the addition or removal of bars.

\subsection{Definition and Interpretation}

In this subsection we recall the definition of the \emph{Topological Stability Index} (TSI), a scalar summary of a persistence barcode defined in terms of the dispersion of its lifetimes.

\begin{definition}
Let $B=\{[b_i,d_i)\}_{i=1}^{n^B}$ be a persistence barcode, and let
\[
\mathcal{L}_B=\{\ell_i=d_i-b_i\}_{i=1}^{n^B}
\]
denote the multiset of bar lengths. The \emph{Topological Stability Index} of $B$ is defined by
\begin{equation}
    \tsi(B):=\var(\mathcal{L}_B),
\end{equation}
where $\var$ denotes the (unbiased) sample variance of the multiset of lifetimes.
\end{definition}

Thus, the TSI measures the dispersion of persistence lifetimes. Small values correspond to barcodes whose lifetimes are nearly uniform, whereas large values indicate stronger heterogeneity, for example due to the presence of dominant or highly variable topological features.

\begin{remark}\rm 
Throughout this paper, we use the unbiased sample variance. Hence, for a barcode $B$ with $n^B\geq 2$ bars,
\begin{equation}
    \tsi(B)=\frac{1}{n^B-1}\sum_{i=1}^{n^B}\left(\ell_i-\frac{1}{n^B}L^B\right)^2,
\end{equation}
where
\[
L^B:=\sum_{i=1}^{n^B}\ell_i
\]
denotes the total persistence. Equivalently,
\begin{equation}
    \tsi(B)=\frac{1}{n^B-1}\left(\sum_{i=1}^{n^B}\ell_i^2-\frac{(L^B)^2}{n^B}\right).
\end{equation}

For convenience, we set $\tsi(B)=0$ when $n^B\leq 1$, including the empty barcode.
\end{remark}



\begin{example}
Consider two barcodes $B_1$ and $B_2$ with lifetime multisets
\[
\mathcal{L}_{B_1}=\{1,1,1,1\},
\qquad
\mathcal{L}_{B_2}=\{0,0,0,4\}.
\]
Both barcodes have the same total persistence,
\[
L^{B_1}=L^{B_2}=4,
\]
but their TSI values are very different. Indeed,
\[
\tsi(B_1)=0,
\]
since all bars in $B_1$ have the same length, while
\[
\tsi(B_2)=4.
\]
Thus, although the two barcodes have the same total persistence, the TSI distinguishes between a perfectly uniform distribution of lifetimes and a highly uneven one. This illustrates that the TSI captures dispersion rather than magnitude alone.
\end{example}

This contrast is illustrated schematically in Figure~\ref{fig:tsi_same_total_persistence}, where two barcodes with identical total persistence exhibit markedly different distributions of lifetimes. The figure emphasizes how concentration of persistence in a few dominant bars leads to a larger TSI compared to a more uniform distribution.

\begin{figure}[H]
\centering
\begin{subfigure}[t]{0.5\textwidth}
        \centering
        \includegraphics[height=1.5in]{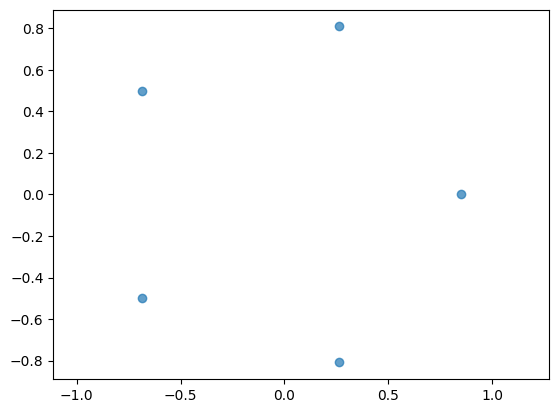}
        \caption{Point cloud.}
    \end{subfigure}
    ~
    \begin{subfigure}[t]{0.3\textwidth}
        \centering
        \includegraphics[height=1.5in]{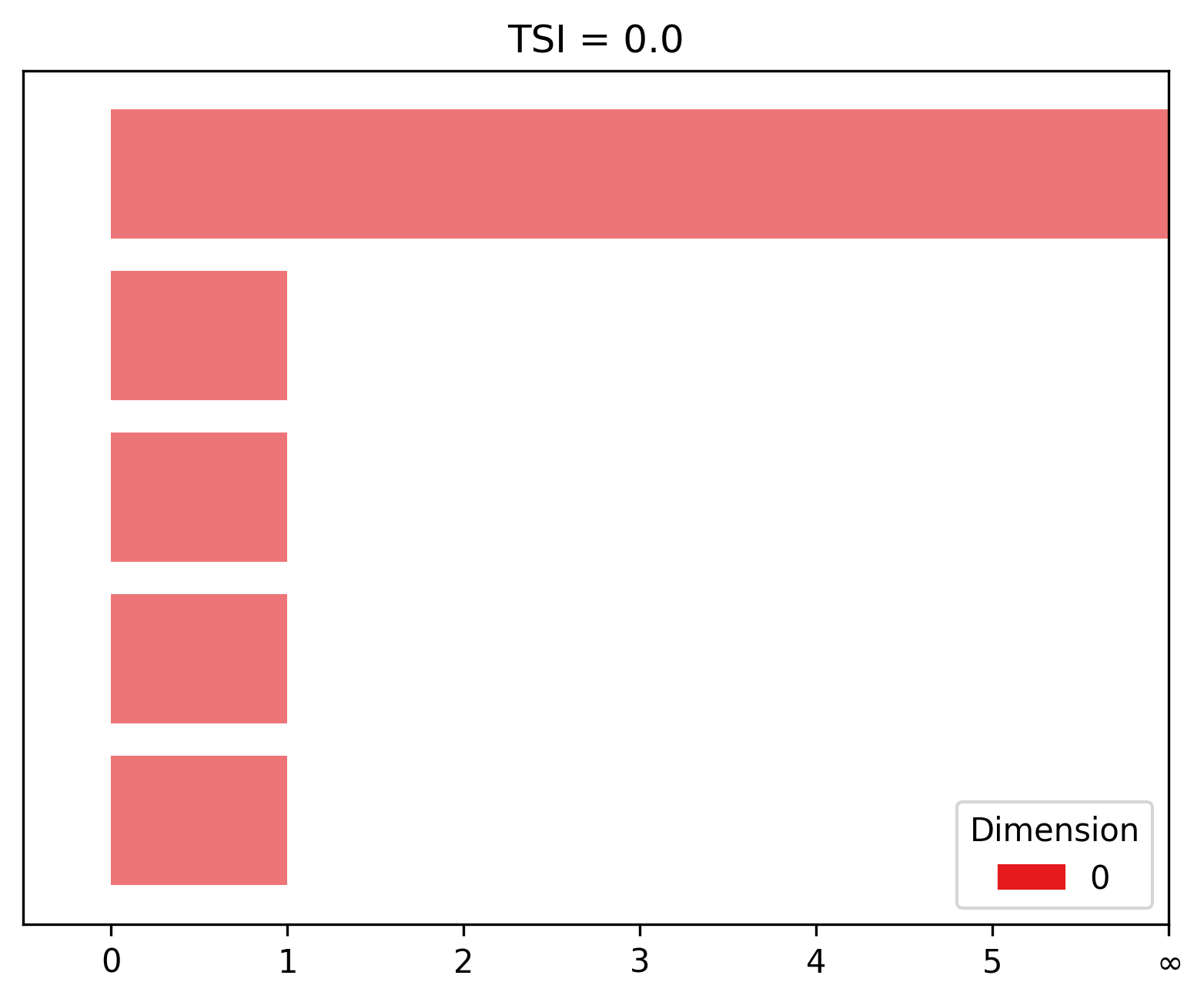}
        \caption{Persistence barcode.}
    \end{subfigure}
    
    \begin{subfigure}[t]{0.5\textwidth}
        \centering
        \includegraphics[height=1.5in]{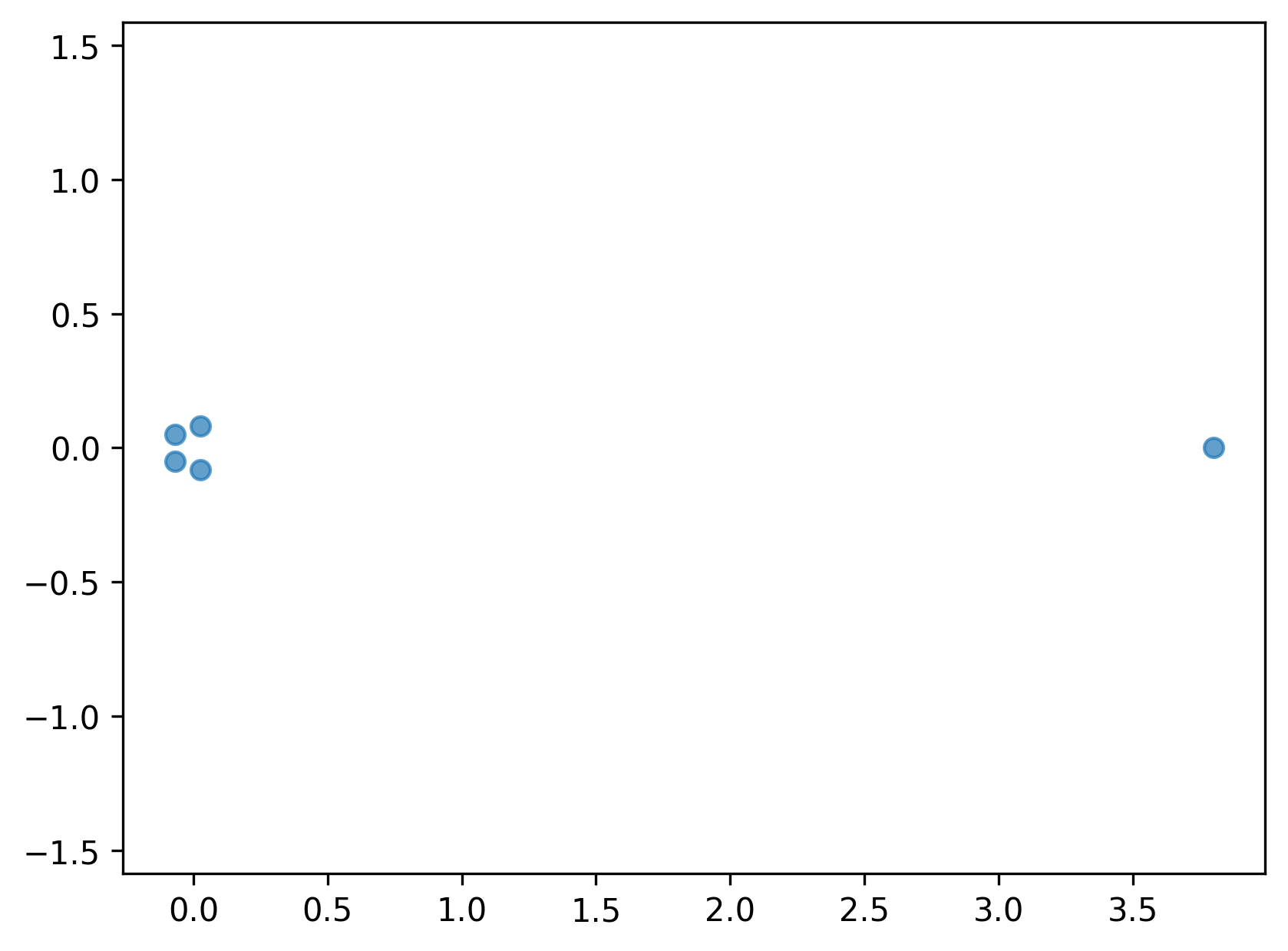}
        \caption{Point cloud.}
    \end{subfigure}
    ~
    \begin{subfigure}[t]{0.3\textwidth}
        \centering
        \includegraphics[height=1.5in]{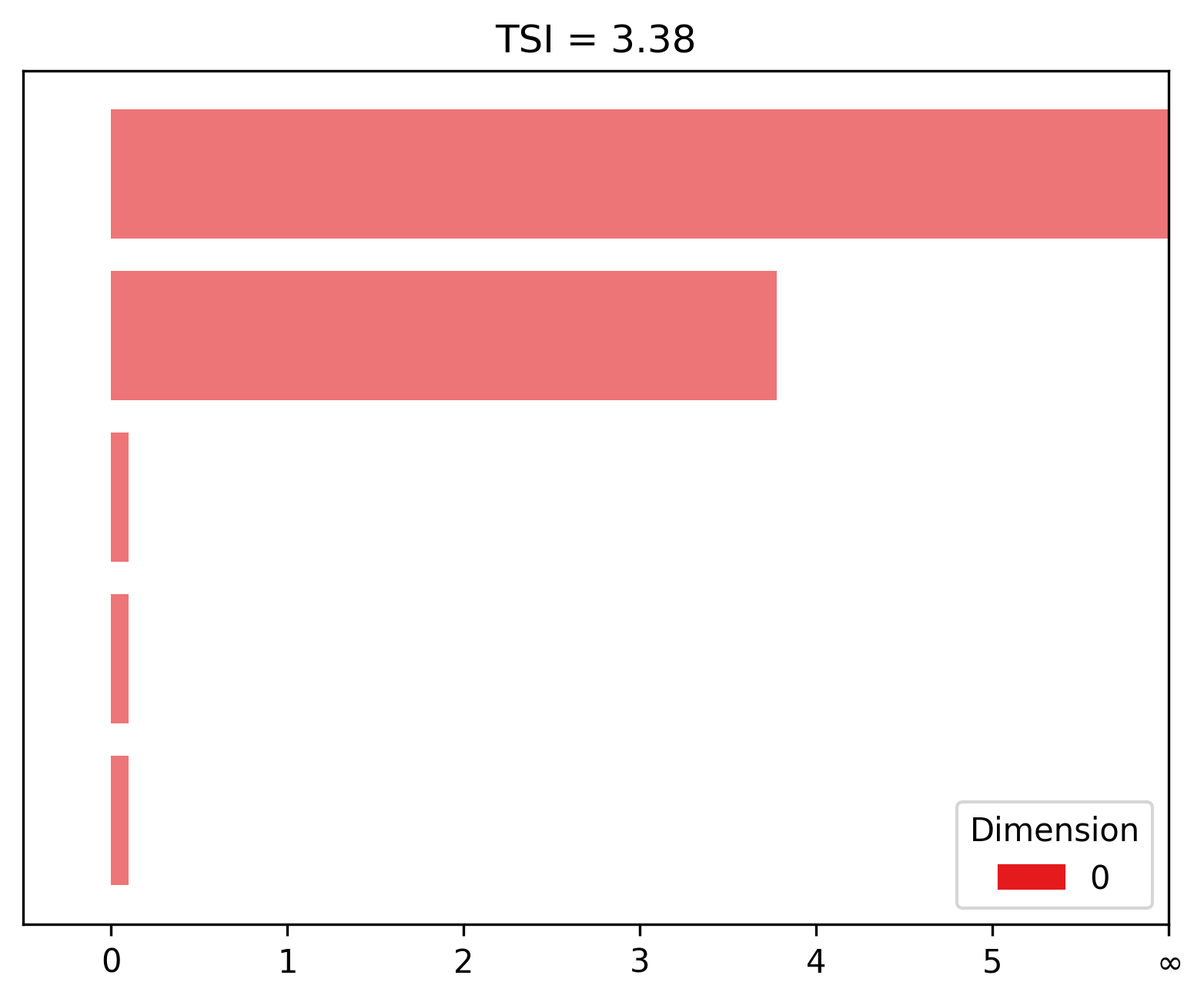}
        \caption{Persistence barcode.}
    \end{subfigure}
\caption{Two persistence barcodes with identical total persistence but different distributions of lifetimes. The top barcode exhibits moderate variability, with one long bar and several shorter ones, whereas the bottom barcode shows a highly uneven distribution, concentrating persistence in a small number of longer bars while the remaining bars are significantly shorter. This highlights that the TSI measures dispersion rather than total persistence.}
\label{fig:tsi_same_total_persistence}
\end{figure}

We now allow barcodes to contain zero-length intervals for the purpose of extremal statements. We have:

\begin{theorem}[Extremal characterization of the TSI]\label{thm:tsi_extremal}
Let $B=\{[b_i,d_i)\}_{i=1}^{n}$ be a barcode with $n\geq 2$ bars and total persistence
\[
L^B=\sum_{i=1}^n \ell_i,
\qquad \ell_i=d_i-b_i\geq 0.
\]
Assume that both $n$ and $L^B$ are fixed.

Then the Topological Stability Index satisfies
\begin{equation}\label{eq:tsi_extremal_bounds}
0\leq \tsi(B)\leq \frac{(L^B)^2}{n}.
\end{equation}
Moreover:
\begin{enumerate}
    \item $\tsi(B)=0$ if and only if all bar lengths are equal, i.e.
    \[
    \ell_1=\cdots=\ell_n=\frac{L^B}{n}.
    \]
    \item $\tsi(B)=\frac{(L^B)^2}{n}$ if and only if, up to permutation of the bars,
    \[
    (\ell_1,\ldots,\ell_n)=(L^B,0,\ldots,0).
    \]
\end{enumerate}
In particular, among all barcodes with fixed number of bars and fixed total persistence, the TSI is minimized by the most uniform barcode and maximized by the most concentrated one.
\end{theorem}

\begin{proof}
Let $n=n^B$ and $L=L^B$. Since the TSI is defined using the unbiased sample variance, we have
\begin{equation}\label{eq:tsi_second_moment_form}
\tsi(B)=\frac{1}{n-1}\left(\sum_{i=1}^n \ell_i^2-\frac{L^2}{n}\right).
\end{equation}
Thus, for fixed $n$ and $L$, the behavior of $\tsi(B)$ is completely determined by the quantity
\[
\sum_{i=1}^n \ell_i^2.
\]

We first determine its minimum under the constraints
\[
\ell_i\geq 0,
\qquad
\sum_{i=1}^n \ell_i=L.
\]

By the Cauchy-Schwarz inequality,
\[
\left(\sum_{i=1}^n \ell_i\right)^2 \leq n\sum_{i=1}^n \ell_i^2.
\]
Since $\sum_i \ell_i=L$, this yields
\[
L^2\leq n\sum_{i=1}^n \ell_i^2,
\]
and therefore
\[
\sum_{i=1}^n \ell_i^2\geq \frac{L^2}{n}.
\]
Substituting into \eqref{eq:tsi_second_moment_form}, we obtain
\[
\tsi(B)\geq \frac{1}{n-1}\left(\frac{L^2}{n}-\frac{L^2}{n}\right)=0.
\]
Hence $\tsi(B)\geq 0$.

Equality holds if and only if equality holds in Cauchy--Schwarz, which happens precisely when all $\ell_i$ are equal. Since their sum is $L$, this means
\[
\ell_1=\cdots=\ell_n=\frac{L}{n}.
\]
This proves the first statement.

We now determine the maximum of $\sum_i \ell_i^2$ subject to the same constraints. Since $\ell_i\geq 0$ and $\sum_i \ell_i=L$, we have
\[
\ell_i^2\leq L\ell_i
\qquad \text{for each } i,
\]
because $0\leq \ell_i\leq L$. Summing over $i$ gives
\[
\sum_{i=1}^n \ell_i^2 \leq L\sum_{i=1}^n \ell_i=L^2.
\]
Substituting this into \eqref{eq:tsi_second_moment_form}, we obtain
\[
\tsi(B)\leq \frac{1}{n-1}\left(L^2-\frac{L^2}{n}\right)
=\frac{1}{n-1}\cdot \frac{n-1}{n}L^2
=\frac{L^2}{n}.
\]
This proves the upper bound in \eqref{eq:tsi_extremal_bounds}.

It remains to characterize when equality holds. Equality in
\[
\sum_{i=1}^n \ell_i^2 \leq L^2
\]
holds if and only if equality holds in $\ell_i^2\leq L\ell_i$ for every $i$. Since
\[
\ell_i^2=L\ell_i
\quad \Longleftrightarrow \quad
\ell_i(\ell_i-L)=0,
\]
each $\ell_i$ must be either $0$ or $L$. Because the sum of all bar lengths is exactly $L$, there can be exactly one index $j$ such that $\ell_j=L$, while all other bar lengths must be zero. Hence, up to permutation,
\[
(\ell_1,\ldots,\ell_n)=(L,0,\ldots,0).
\]
This proves the second statement.

Therefore, for fixed $n^B$ and $L^B$, the TSI is minimized by the most uniform barcode and maximized by the most concentrated barcode.
\end{proof}

The extremal configurations described in Theorem~\ref{thm:tsi_extremal} are illustrated in Figure~\ref{fig:tsi_extremal_configurations}. 
The left panel shows the uniform distribution of lifetimes, which minimizes the TSI, while the right panel shows the fully concentrated configuration, which maximizes it.

\begin{figure}[H]
\centering
\begin{subfigure}[t]{0.5\textwidth}
        \centering
        \includegraphics[height=2in]{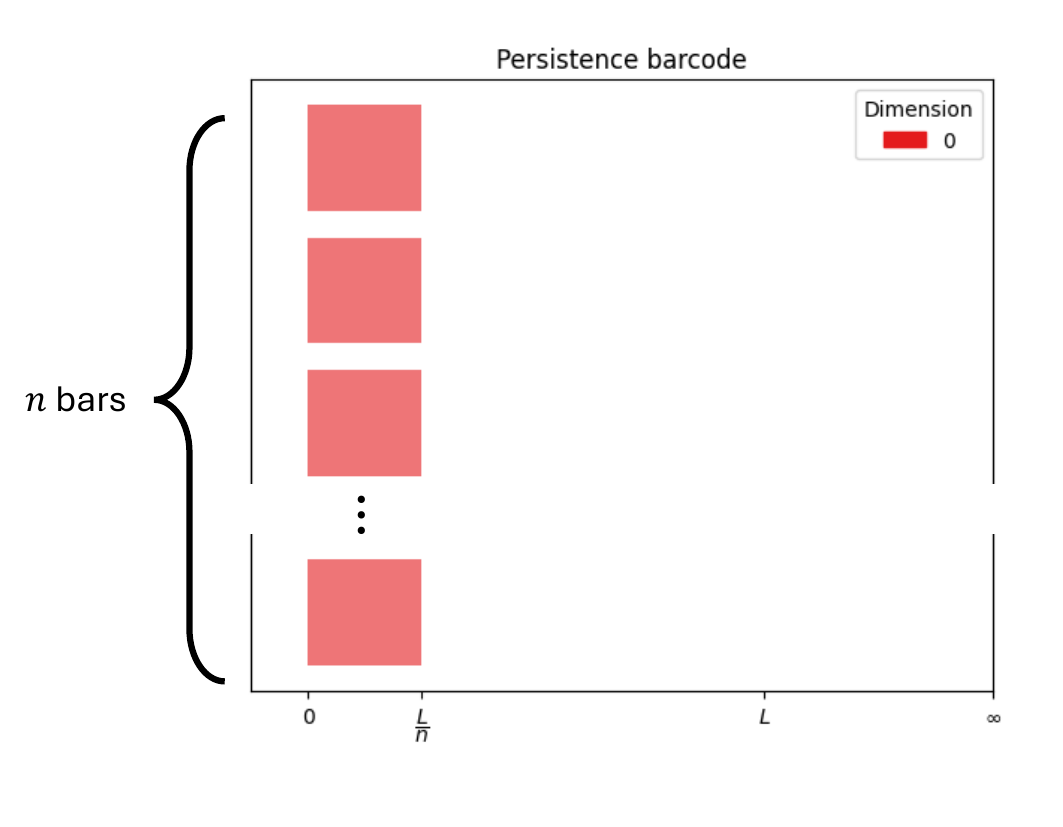}
        \caption{Barcode where $\tsi = 0$.}
    \end{subfigure}
    ~
    \begin{subfigure}[t]{0.5\textwidth}
        \centering
        \includegraphics[height=2in]{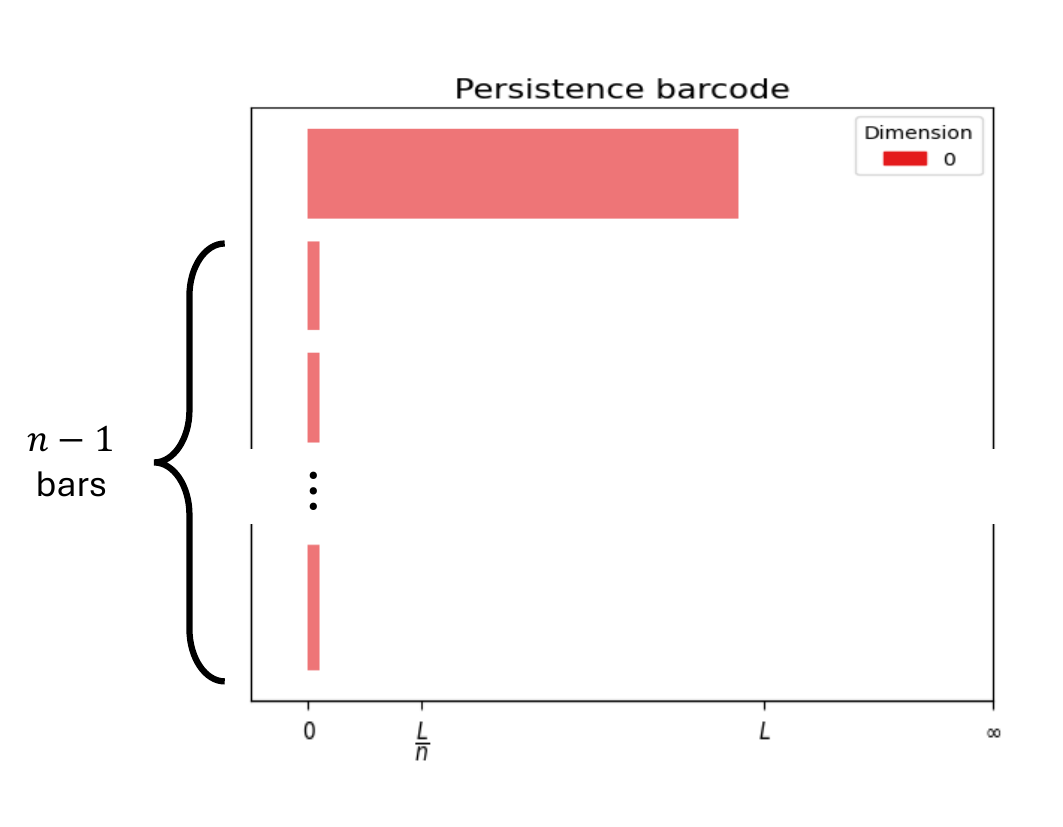}
        \caption{Barcode where TSI is maximal.}
    \end{subfigure}
\caption{Extremal configurations for the TSI under fixed total persistence $L$ and number of bars $n$. A barcode with equal bar lengths $\ell_i = L/n$ minimizes the TSI (left), while a barcode in which one bar has length $L$ and the remaining $n-1$ bars have length zero (shown schematically) maximizes the TSI (right).}
\label{fig:tsi_extremal_configurations}
\end{figure}

\subsection{Behaviour under Scaling and Translation}

We now examine how TSI behaves under simple transformations of the barcode and establish several basic structural properties.

We first consider the effect of scaling all filtration values.

\begin{prop}[Scaling]\label{prop:scalar_mult_tsi}
Let $B$ be a barcode and $c\geq 0$. Define
\[
cB=\{[cb,cd):[b,d)\in B\}.
\]
Then
\begin{equation}
\tsi(cB)=c^2\tsi(B).
\end{equation} 
\end{prop}

\begin{proof}
Using the explicit expression for the sample variance, we obtain
\begin{align*}
\tsi(cB)
&=\frac{1}{n^B-1}\sum_{i=1}^{n^B}\left(c\ell_i-\frac{c}{n^B}L^B\right)^2\\
&=c^2\frac{1}{n^B-1}\sum_{i=1}^{n^B}\left(\ell_i-\frac{1}{n^B}L^B\right)^2
=c^2\tsi(B).
\end{align*}
\end{proof}

While translating the underlying point cloud does not affect persistence diagrams arising from distance-based filtrations, it is natural to consider the effect of translating persistence lifetimes directly. More precisely, consider the transformation that shifts all death times by a constant while keeping birth times fixed (or equivalently shifts all lifetimes by a constant). Such operations arise, for example, when comparing filtrations that differ by a uniform offset. In general, barcode summaries are not invariant under such transformations. However, the TSI remains unchanged, as shown in the following proposition.

\begin{prop}[Translation of lifetimes (death-shift invariance)]\label{prop:length_translate_tsi}

Let $B$ be a barcode and let $c\in\mathbb{R}$ satisfy $c>-\min_i \ell_i$. Then
\begin{equation}
\tsi(\{[b,d+c):[b,d)\in B\})=\tsi(B).
\end{equation}
\end{prop}

\begin{proof}
We compute
\begin{align*}
\tsi(\{[b,d+c)\})
&=\frac{1}{n^B-1}\sum_{i=1}^{n^B}\left(\ell_i+c-\frac{1}{n^B}(L^B+n^Bc)\right)^2\\
&=\frac{1}{n^B-1}\sum_{i=1}^{n^B}\left(\ell_i-\frac{1}{n^B}L^B\right)^2
=\tsi(B).
\end{align*}
\end{proof}

\begin{remark}\rm 
The transformation in Proposition \ref{prop:length_translate_tsi} corresponds to a vertical shift in the persistence diagram and to a uniform translation in lifetime space. Moreover, Propositions \ref{prop:scalar_mult_tsi} and \ref{prop:length_translate_tsi} show that the TSI is invariant under translations of lifetimes and homogeneous of degree two under uniform scaling. In particular, the TSI depends only on deviations of lifetimes from their mean and scales quadratically under uniform rescaling. These effects are illustrated in Figure~\ref{fig:tsi_scaling_translation}, where scaling increases the TSI while a uniform shift leaves it unchanged.
\end{remark}

\begin{figure}[htbp]
\centering
\begin{subfigure}[t]{0.33\textwidth}
    \centering
    \includegraphics[height=1.75in]{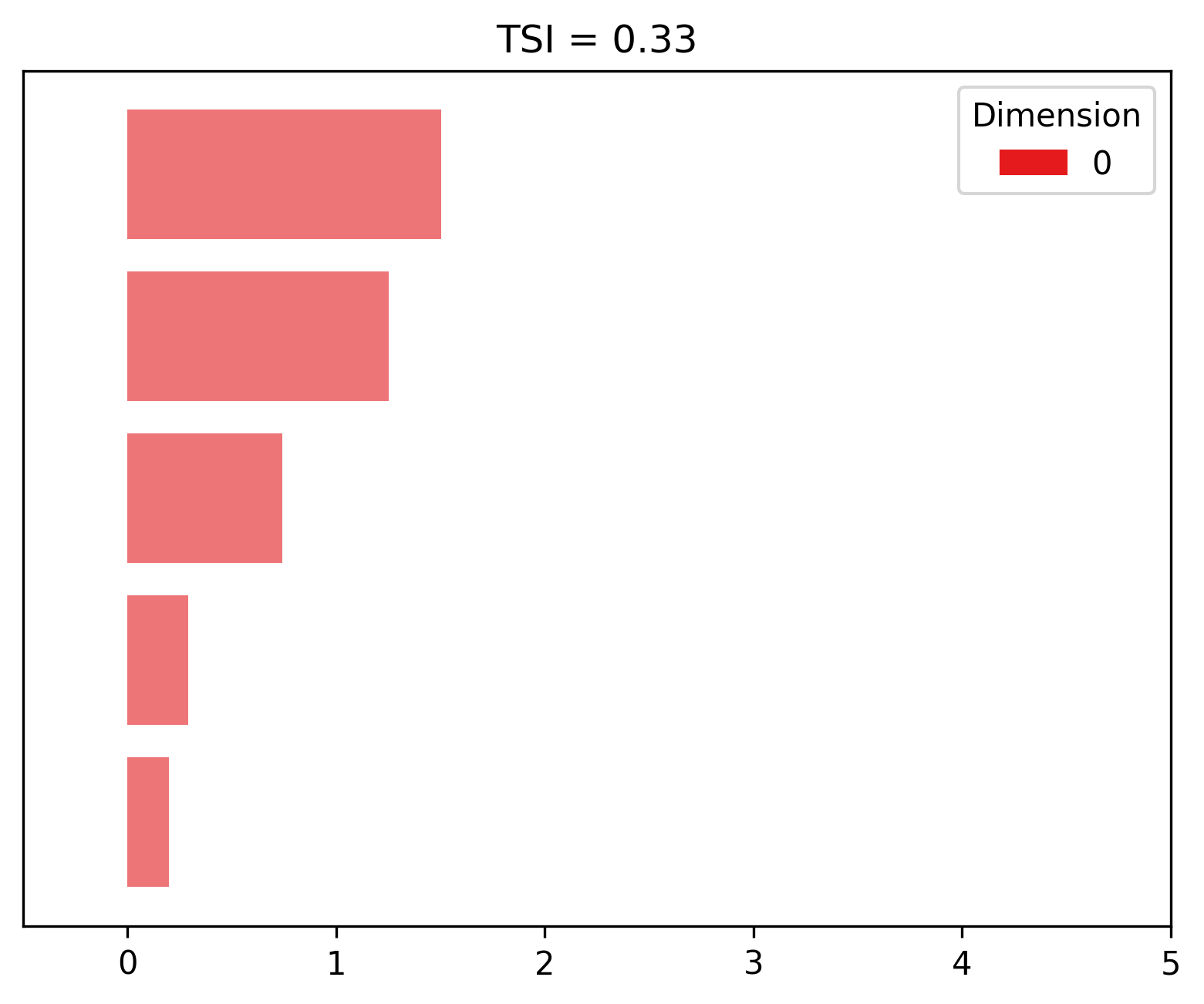}
    \caption{Example of a barcode.}
\end{subfigure}
~
\begin{subfigure}[t]{0.33\textwidth}
    \centering
    \includegraphics[height=1.75in]{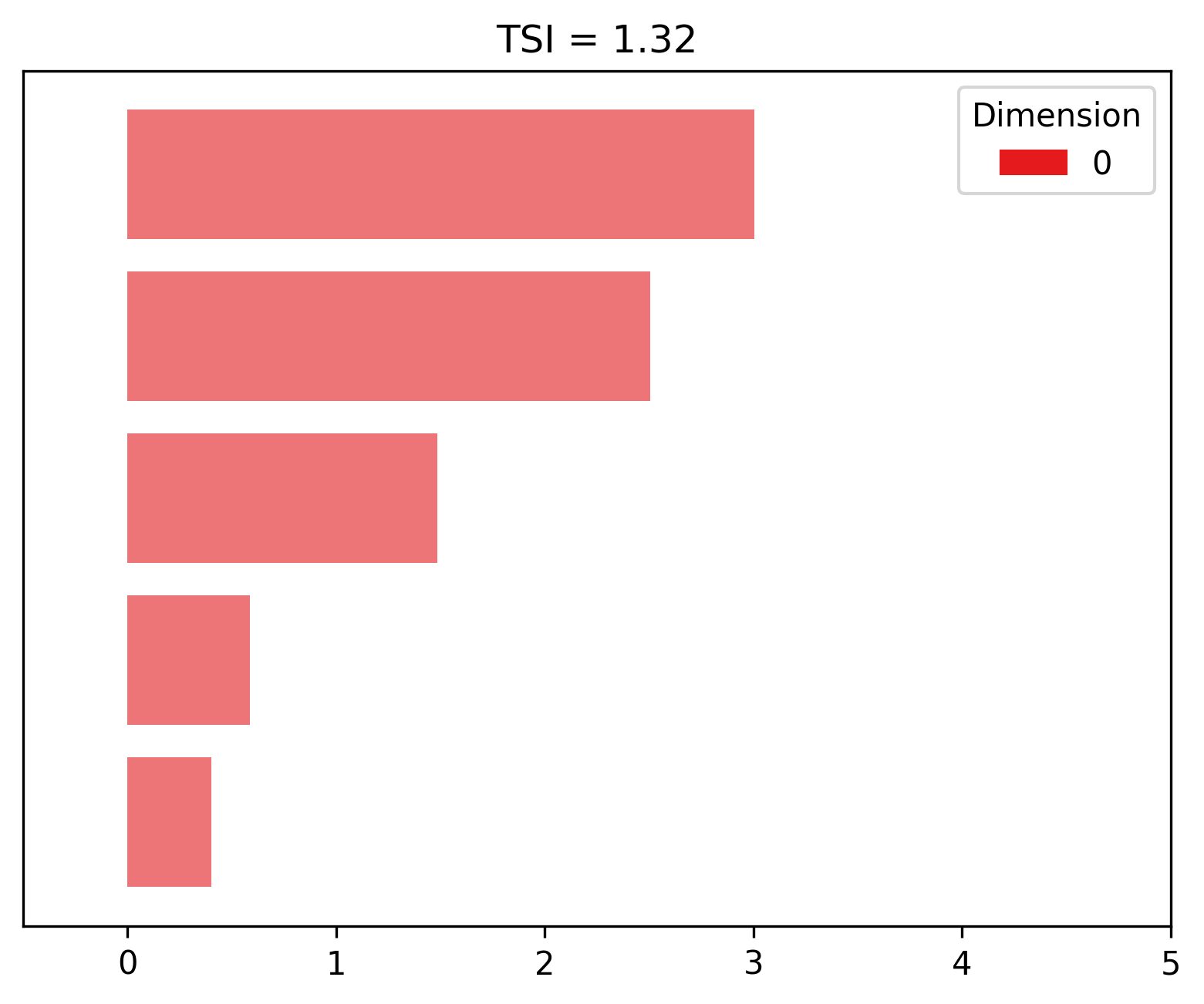}
    \caption{Barcode scaled by 2.}
\end{subfigure}
~
\begin{subfigure}[t]{0.33\textwidth}
    \centering
    \includegraphics[height=1.75in]{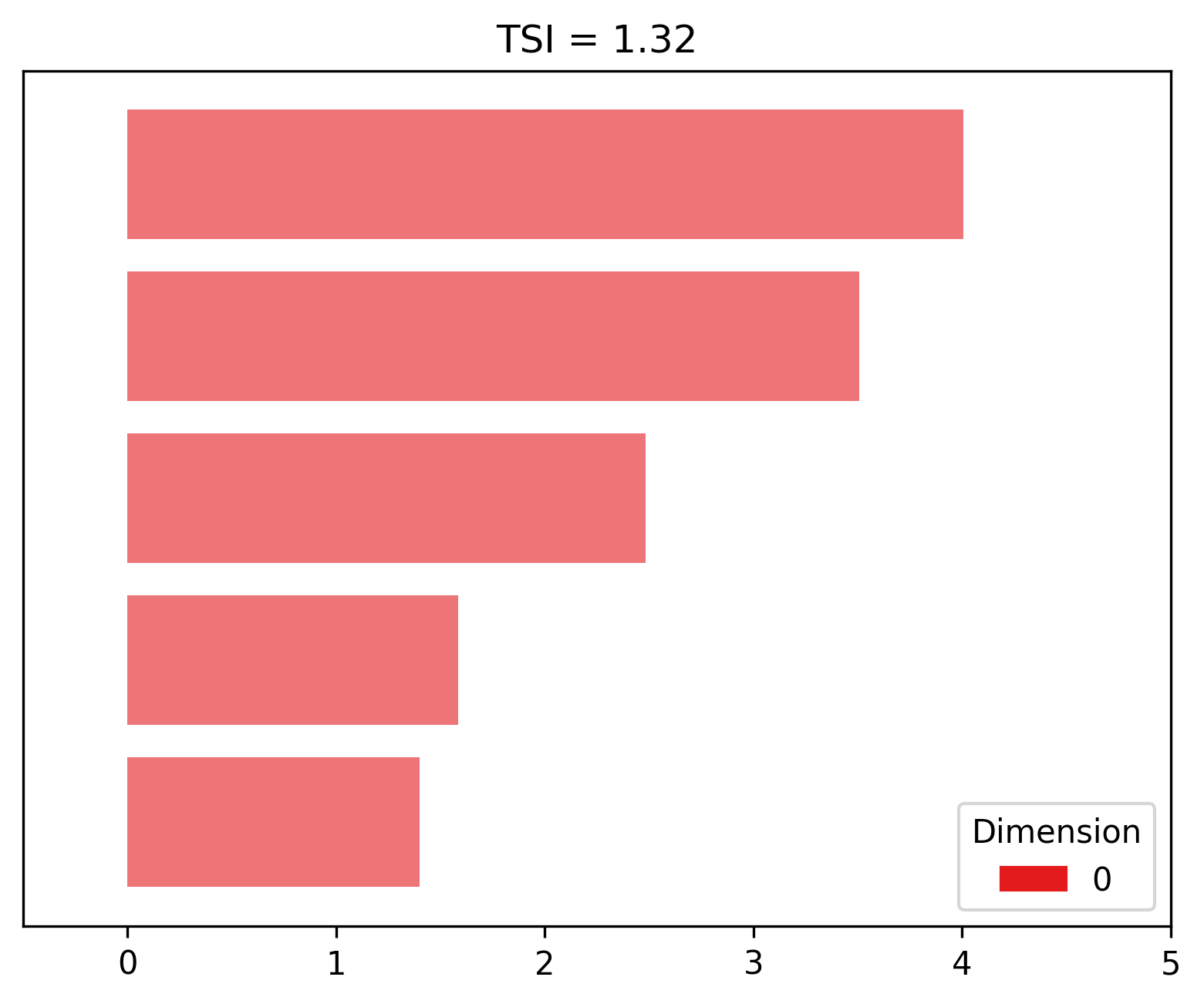}
    \caption{Death times increased by 1.}
\end{subfigure}

\caption{Effect of uniform scaling and death shifts on a barcode. Uniform scaling multiplies all lifetimes by the same factor and causes the TSI to scale quadratically, whereas shifting all death times by a constant leaves the TSI unchanged.}
\label{fig:tsi_scaling_translation}
\end{figure}

\subsection{Addition and Removal of Bars}

We now describe how the TSI changes under insertion or deletion of a bar. To do this, we need the following lemma.

\begin{lemma}\label{lemma:1_changed_bar}
Let $B$ be a barcode with $n^B\geq 2$ bars, and let
\[
\bar{\ell}_B:=\frac{L^B}{n^B}
\]
denote the mean bar length.

Adding a bar of length $\ell$ yields
\begin{equation}
\tsi(B\cup\{[b,b+\ell)\})
=\frac{n^B-1}{n^B}\tsi(B)
+\frac{1}{n^B+1}\left(\ell-\bar{\ell}_B\right)^2.
\end{equation}

If moreover $n^B\geq 3$, then removing a bar of length $\ell$ gives
\begin{equation}
\tsi(B\setminus\{[b,b+\ell)\})
=\frac{n^B-1}{n^B-2}\tsi(B)
-\frac{n^B}{(n^B-1)(n^B-2)}\left(\ell-\bar{\ell}_B\right)^2.
\end{equation}
\end{lemma}

\begin{proof}
Write $n=n^B$, let the bar lengths of $B$ be $\ell_1,\dots,\ell_n$, and let
\[
\bar{\ell}=\frac{L^B}{n}=\frac{1}{n}\sum_{i=1}^n \ell_i.
\]
By definition,
\[
(n-1)\tsi(B)=\sum_{i=1}^n (\ell_i-\bar{\ell})^2.
\]

We first prove the formula for adding a bar of length $\ell$.
Let
\[
B^+:=B\cup\{[b,b+\ell)\}.
\]
Its new mean bar length is
\[
\bar{\ell}^+=\frac{L^B+\ell}{n+1}
=\bar{\ell}+\frac{\ell-\bar{\ell}}{n+1}.
\]
Set $\delta:=\ell-\bar{\ell}$. Then
\[
\bar{\ell}^+=\bar{\ell}+\frac{\delta}{n+1}.
\]
Hence
\begin{align*}
n\,\tsi(B^+)
&=\sum_{i=1}^n (\ell_i-\bar{\ell}^+)^2+(\ell-\bar{\ell}^+)^2\\
&=\sum_{i=1}^n \left((\ell_i-\bar{\ell})-\frac{\delta}{n+1}\right)^2
+\left(\delta-\frac{\delta}{n+1}\right)^2.
\end{align*}
Expanding the first sum and using
\[
\sum_{i=1}^n (\ell_i-\bar{\ell})=0,
\]
we get
\begin{align*}
n\,\tsi(B^+)
&=\sum_{i=1}^n (\ell_i-\bar{\ell})^2
+\frac{n\delta^2}{(n+1)^2}
+\frac{n^2\delta^2}{(n+1)^2}\\
&=(n-1)\tsi(B)+\frac{n(n+1)\delta^2}{(n+1)^2}\\
&=(n-1)\tsi(B)+\frac{n}{n+1}\delta^2.
\end{align*}
Dividing by \(n\) yields
\[
\tsi(B^+)
=\frac{n-1}{n}\tsi(B)+\frac{1}{n+1}\delta^2,
\]
which proves the first formula.

For the removal formula, let \(B^-\) be obtained from \(B\) by deleting a bar of length \(\ell\), and again write \(\delta=\ell-\bar{\ell}\).
The new mean is
\[
\bar{\ell}^-=\frac{L^B-\ell}{n-1}
=\bar{\ell}-\frac{\delta}{n-1}.
\]
Thus
\begin{align*}
(n-2)\tsi(B^-)
&=\sum_{i\neq k}(\ell_i-\bar{\ell}^-)^2,
\end{align*}
where \(\ell_k=\ell\) is the removed bar.
Using
\[
(n-1)\tsi(B)=\sum_{i\neq k}(\ell_i-\bar{\ell})^2+(\ell-\bar{\ell})^2
\]
and expanding with \(\bar{\ell}^-=\bar{\ell}-\delta/(n-1)\), one finds
\[
(n-2)\tsi(B^-)
=(n-1)\tsi(B)-\frac{n}{n-1}\delta^2.
\]
Therefore,
\[
\tsi(B^-)
=\frac{n-1}{n-2}\tsi(B)-\frac{n}{(n-1)(n-2)}\delta^2,
\]
which proves the second formula.
\end{proof}

\begin{remark}\rm
Lemma \ref{lemma:1_changed_bar} provides an exact finite-sample update formula for the TSI, analogous to classical incremental variance formulas in statistics. In particular, it shows that the TSI is sensitive not only to the magnitude of a new bar but to its deviation from the existing mean lifetime.
\end{remark}

\begin{corollary}
Under the notation of Lemma \ref{lemma:1_changed_bar}, adding a bar of length \(\ell\) increases the TSI if and only if
\[
\left(\ell-\bar{\ell}_B\right)^2>\frac{n^B+1}{n^B}\tsi(B).
\]
\end{corollary}

\begin{remark}\rm
The threshold
\[
\sqrt{\frac{n^B+1}{n^B}\tsi(B)}
\]
acts as a variance barrier: bars further from the mean than this value increase dispersion, while closer bars reduce it.
\end{remark}

This threshold behavior is illustrated in Figure~\ref{fig:tsi_bar_insertion}, where the dashed line indicates the mean lifetime $\bar{\ell}_B$.

\begin{figure}[htbp]
\centering
\begin{subfigure}[t]{0.33\textwidth}
    \centering
    \includegraphics[height=1.75in]{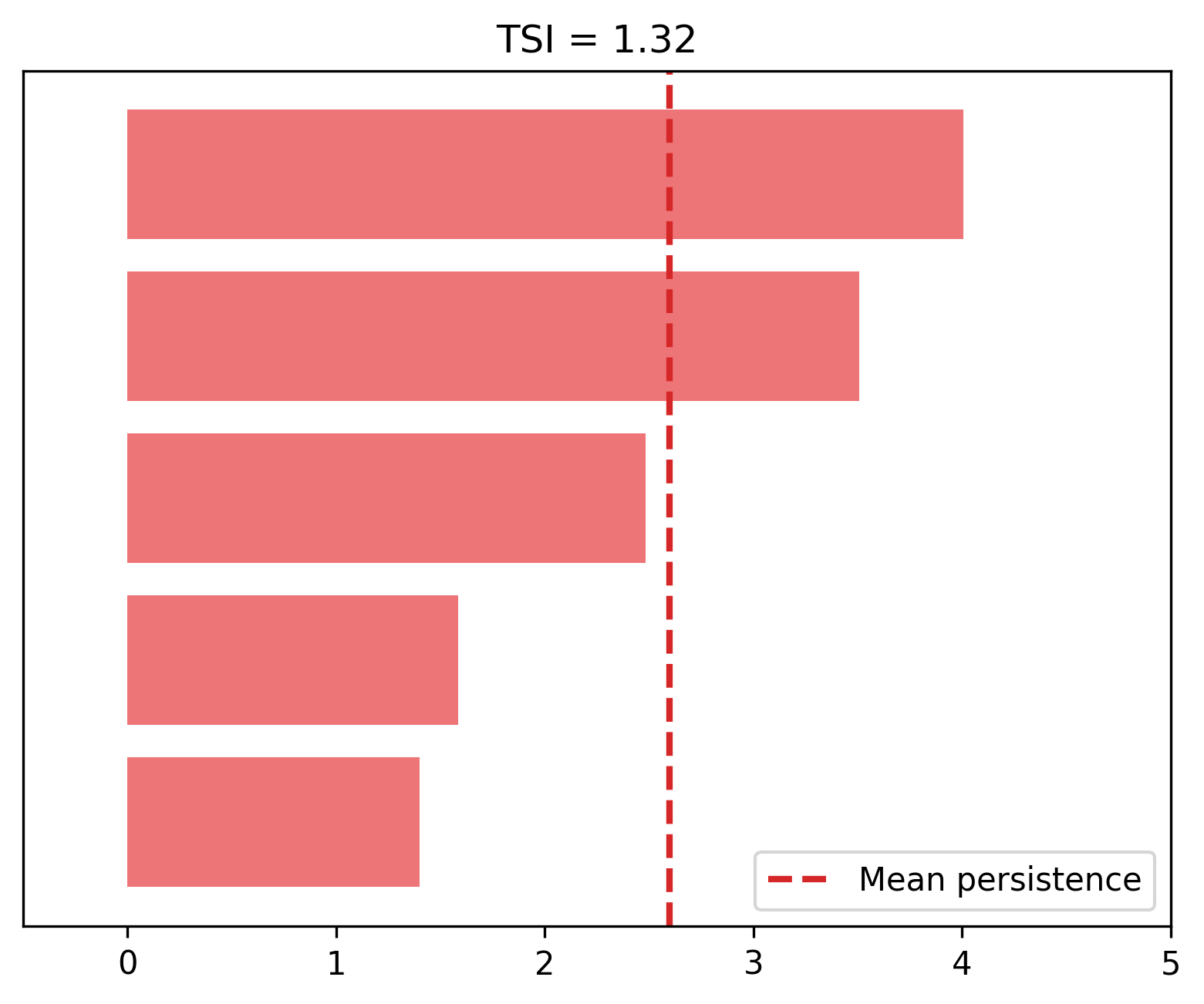}
    \caption{Example of a barcode with mean.}
\end{subfigure}
~
\begin{subfigure}[t]{0.33\textwidth}
    \centering
    \includegraphics[height=1.75in]{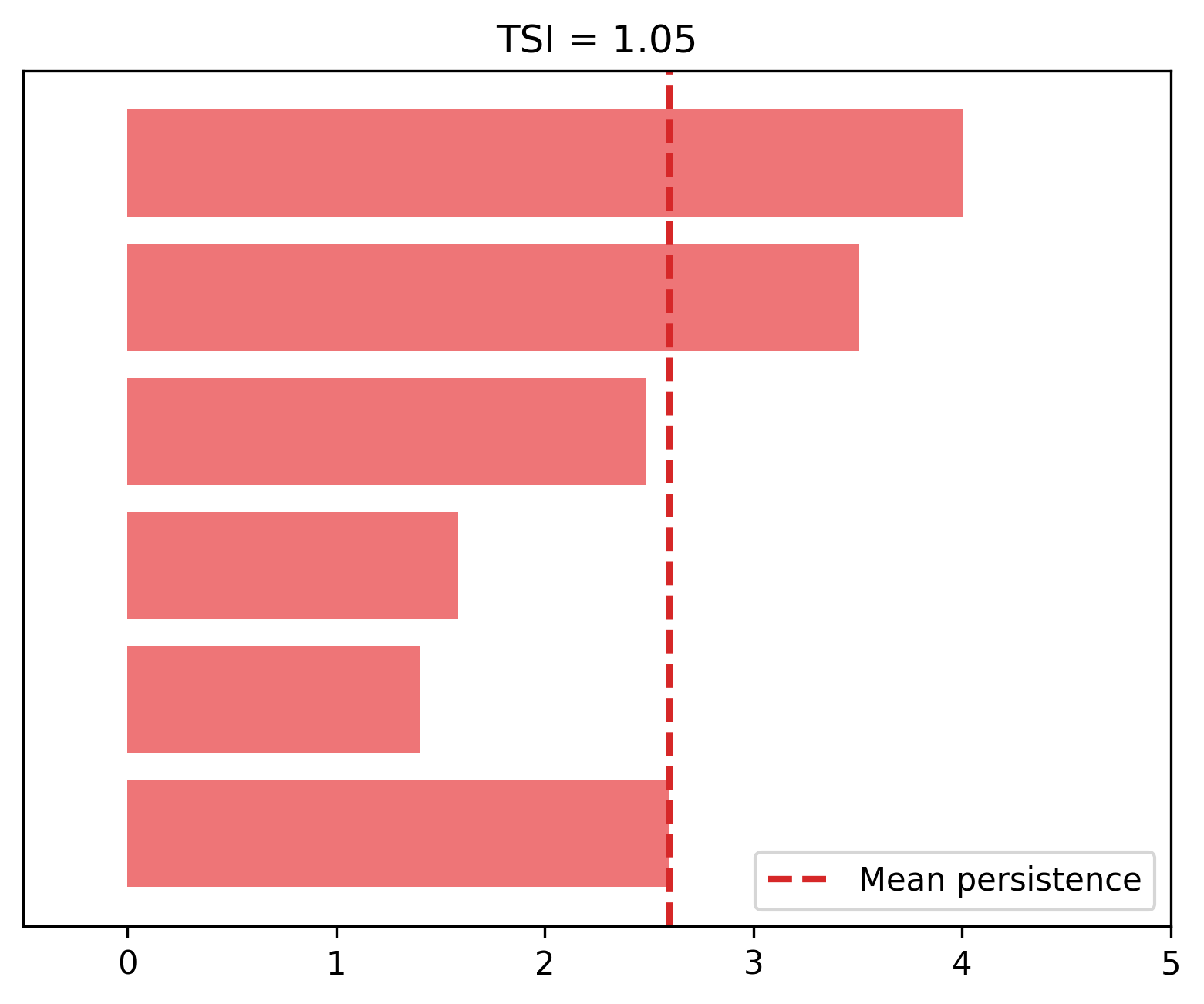}
    \caption{Adding a bar at the mean.}
\end{subfigure}
~
\begin{subfigure}[t]{0.33\textwidth}
    \centering
    \includegraphics[height=1.75in]{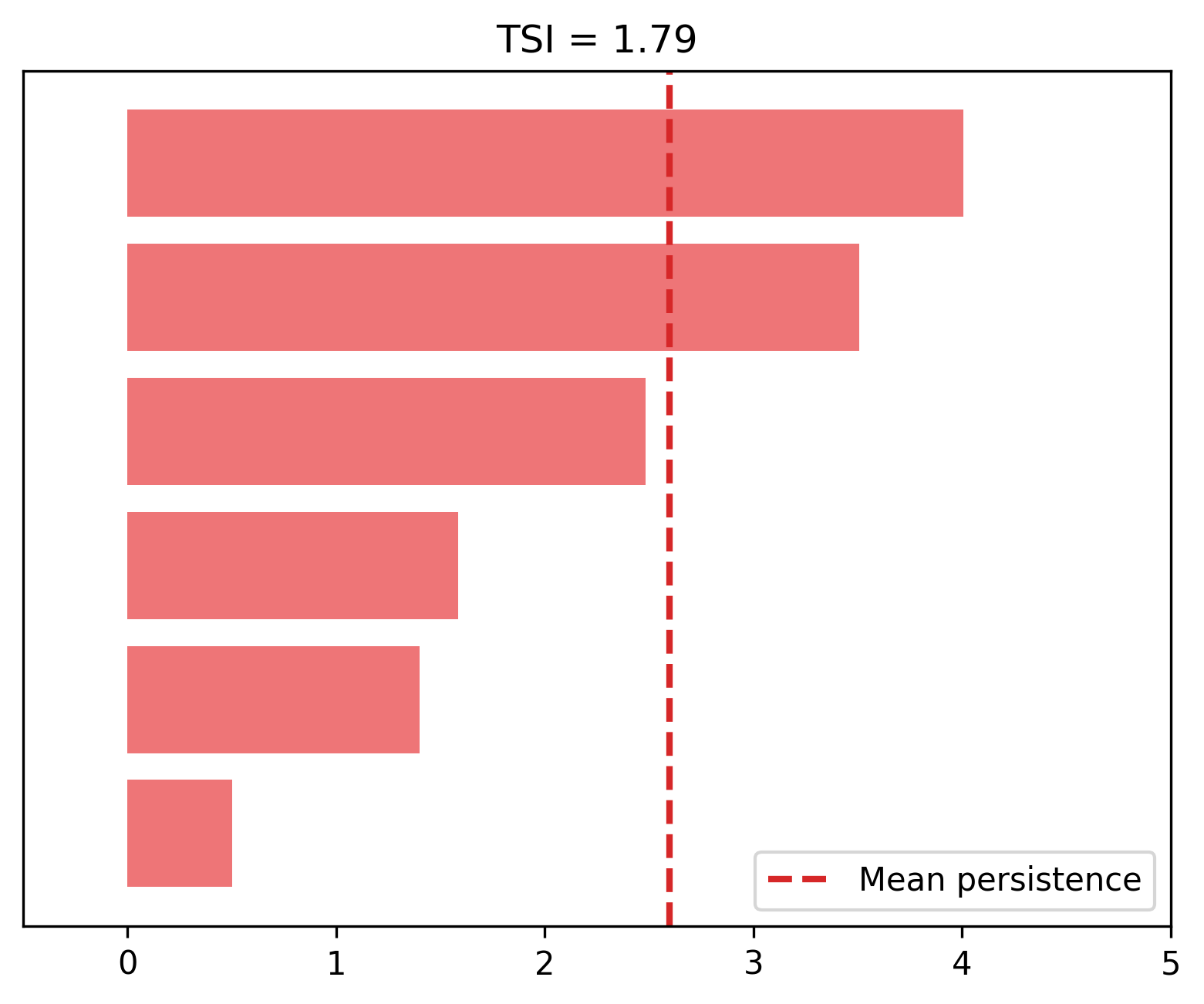}
    \caption{Adding a bar away from the mean.}
\end{subfigure}
\caption{Effect of inserting a new bar. If the new bar length is close to the mean lifetime, the TSI may decrease; if it is sufficiently far from the mean, the TSI increases. This threshold behavior is quantified by the variance barrier in the preceding corollary.}
\label{fig:tsi_bar_insertion}
\end{figure}

This sensitivity to the position of the inserted bar has an important consequence for stability. As a consequence, if $\ell\to 0$,
\[
\tsi(B\cup\{[b,b+\ell)\})
\to
\frac{n^B-1}{n^B}\tsi(B)
+\frac{1}{n^B+1}\left(\frac{L^B}{n^B}\right)^2,
\]
which in general differs from $\tsi(B)$. Thus, TSI is not continuous under the insertion of arbitrarily short bars. In particular, no Lipschitz-type stability bound of the form
\[
|\tsi(B)-\tsi(B')|\le C\,d_{W_p}(B,B')
\]
with a constant $C$ independent of the number of bars can hold. Indeed, adding a bar of length $\ell$ contributes a quantity of order $\ell$ to the Wasserstein distance, while the TSI converges to a value different from $\tsi(B)$ as $\ell\to 0$, because the insertion changes both the sample mean and the sample-size normalization. This behavior is illustrated in Figure~\ref{fig:tsi_noncontinuity}.

\begin{figure}[htbp]
\centering
\includegraphics[width = 0.7\textwidth]{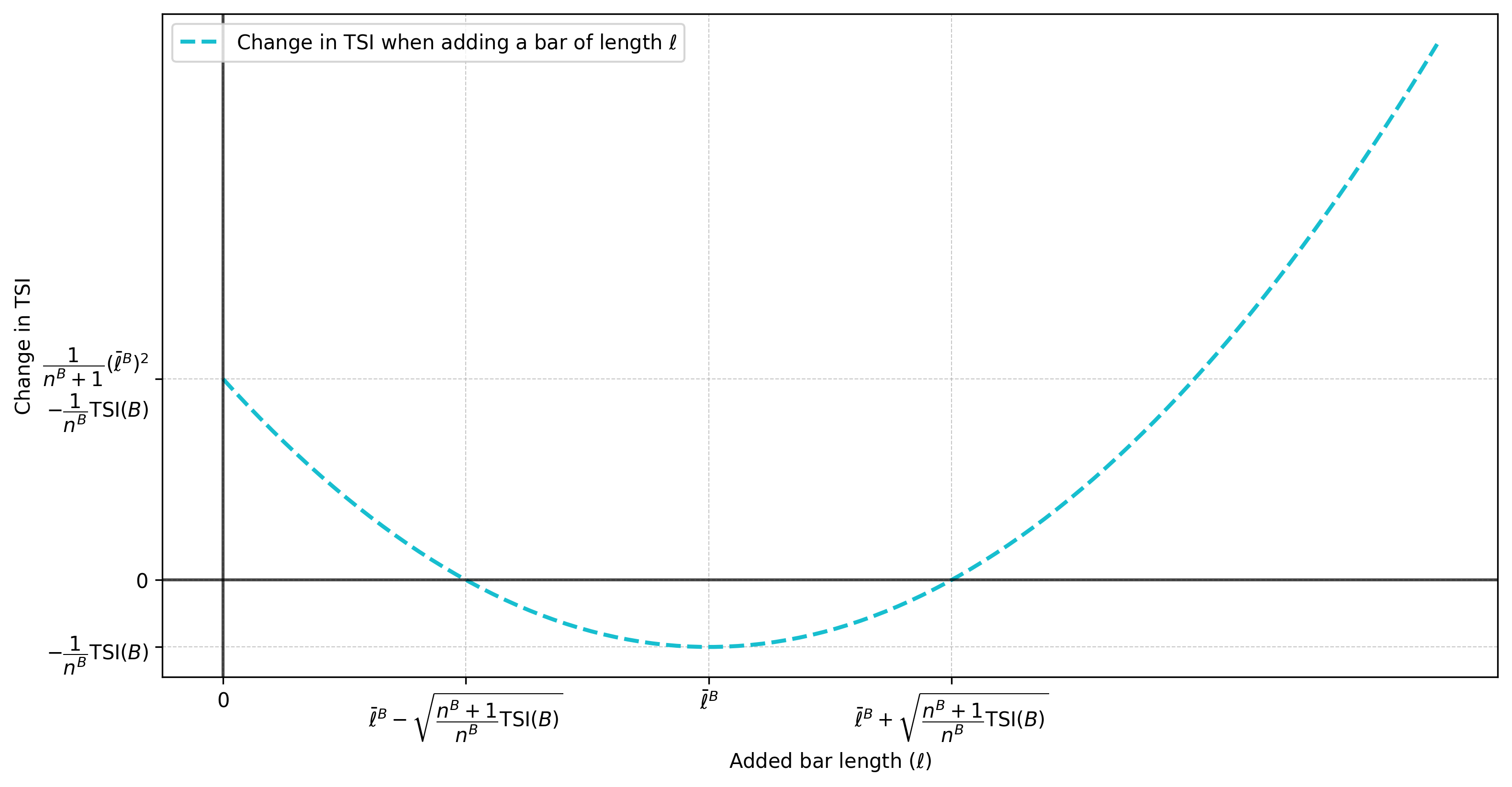}
\caption{Non-continuity of the TSI under insertion of arbitrarily short bars. As the added bar length $\ell$ tends to zero, the Wasserstein distance to the original barcode vanishes, while the TSI converges to a value that in general differs from the original one.}
\label{fig:tsi_noncontinuity}
\end{figure}


\subsection{Stability Properties and Quantitative Bounds}

Despite the lack of a general stability bound under the insertion of bars, meaningful estimates can still be obtained in specific settings. In particular, one can control the TSI when comparing a barcode to the empty diagram, where all intervals are matched to the diagonal. This leads to the following bound.

\begin{lemma}[Bounds relative to the empty diagram]
Let $B=\{[b_i,d_i)\}_{i=1}^{n^B}$ be a barcode. Then, for every $2\le p\le \infty$, one has
\begin{equation}
\tsi(B)\le 4 (n^B)^{-2/p}\, d_{W_p}(B,\emptyset)^2,
\end{equation}
where the Wasserstein distance to the empty diagram is computed by matching all intervals to the diagonal.
\end{lemma}

\begin{proof}
Write $n=n^B$ and let $\ell_i=d_i-b_i\ge 0$ denote the bar lengths. Recall that
\begin{equation}\label{eq:tsi_variance_formula_empty}
\tsi(B)=\frac{1}{n-1}\left(\sum_{i=1}^n \ell_i^2-\frac{1}{n}\left(\sum_{i=1}^n \ell_i\right)^2\right).
\end{equation}
Also, when matching to the diagonal, one has
\[
d_{W_p}(B,\emptyset)
=
\left(\sum_{i=1}^n \left(\frac{\ell_i}{2}\right)^p\right)^{1/p}
=
\frac12\left(\sum_{i=1}^n \ell_i^p\right)^{1/p}
\qquad (2\le p<\infty),
\]
and
\[
d_{W_\infty}(B,\emptyset)=\frac12\max_i \ell_i.
\]

We first treat the case $2\le p<\infty$. Let
\[
\|\ell\|_p:=\left(\sum_{i=1}^n \ell_i^p\right)^{1/p}.
\]
Then \eqref{eq:tsi_variance_formula_empty} can be written as
\[
\tsi(B)=\frac{1}{n-1}\left(\|\ell\|_2^2-\frac{1}{n}\|\ell\|_1^2\right).
\]
By the standard norm inequalities on $\mathbb{R}^n$, for $2\le p<\infty$ we have
\[
\|\ell\|_2 \le n^{\frac12-\frac1p}\|\ell\|_p
\qquad\text{and}\qquad
\|\ell\|_1 \ge \|\ell\|_2.
\]
Hence
\[
\|\ell\|_1^2 \ge \|\ell\|_2^2,
\]
so from \eqref{eq:tsi_variance_formula_empty} we obtain
\begin{align*}
\tsi(B)
&\le \frac{1}{n-1}\left(\|\ell\|_2^2-\frac{1}{n}\|\ell\|_2^2\right)\\
&=\frac{1}{n}\|\ell\|_2^2.
\end{align*}
Using again $\|\ell\|_2 \le n^{\frac12-\frac1p}\|\ell\|_p$, we get
\[
\tsi(B)\le \frac{1}{n}\, n^{1-\frac{2}{p}}\|\ell\|_p^2
= n^{-2/p}\|\ell\|_p^2.
\]
Since
\[
\|\ell\|_p = 2\, d_{W_p}(B,\emptyset),
\]
it follows that
\[
\tsi(B)\le 4 n^{-2/p}\, d_{W_p}(B,\emptyset)^2.
\]

Now consider the case $p=\infty$. Since
\[
\|\ell\|_2^2=\sum_{i=1}^n \ell_i^2 \le n\,\|\ell\|_\infty^2,
\]
equation \eqref{eq:tsi_variance_formula_empty} gives
\[
\tsi(B)\le \frac{1}{n-1}\|\ell\|_2^2 \le \frac{n}{n-1}\|\ell\|_\infty^2.
\]
A sharper estimate is obtained by combining
\[
\tsi(B)\le \frac{1}{n}\|\ell\|_2^2
\]
with
\[
\|\ell\|_2^2\le n\,\|\ell\|_\infty^2,
\]
which yields
\[
\tsi(B)\le \|\ell\|_\infty^2.
\]
Since
\[
\|\ell\|_\infty = 2\, d_{W_\infty}(B,\emptyset),
\]
we conclude that
\[
\tsi(B)\le 4\, d_{W_\infty}(B,\emptyset)^2.
\]
This is exactly the claimed bound, since $n^{-2/\infty}=n^0=1$.
\end{proof}

The previous lemma shows that, although the TSI is not stable under arbitrary barcode perturbations, it is still controlled when all bars are compared directly to the diagonal. In this sense, the TSI is controlled by the squared $p$-moment of the lifetime distribution when the barcode is compared to the empty diagram.

\begin{remark}\rm
For the bottleneck distance ($p=\infty$), a bound in terms of the range of the lifetimes follows from Popoviciu's inequality
\cite{bhatia_better_2000}. Writing
\[
L_\infty^B:=\max_i \ell_i,
\qquad
L_0^B:=\min_i \ell_i,
\]
and recalling that \(\tsi\) is the unbiased sample variance, we obtain
\[
\tsi(B)\leq
\frac{n^B}{n^B-1}\frac{1}{4}(L_\infty^B-L_0^B)^2.
\]
Since
\[
d_{W_\infty}(B,\emptyset)=\frac12 L_\infty^B,
\]
we also have
\[
\tsi(B)\leq
\frac{n^B}{n^B-1}d_{W_\infty}(B,\emptyset)^2.
\]
Since \(\tsi(\emptyset)=0\), this yields a bound on
\(|\tsi(B)-\tsi(\emptyset)|\) in terms of the bottleneck distance.
\end{remark}

Finally, assuming two barcodes have the same number of bars, the following lemma describes the relationship between TSI and bottleneck distance:

\begin{lemma}
Let $B_1$ and $B_2$ be barcodes with the same number of bars $n\geq 2$. Then
\begin{equation}
|\tsi(B_1)-\tsi(B_2)|
\leq \frac{4}{n-1}\left(L^{B_1}+L^{B_2}\right)d_B(B_1,B_2).
\end{equation}
\end{lemma}

\begin{proof}
The proof follows ideas from \cite{atienza_persistent_2019}.
Let $\delta=d_B(B_1,B_2)$ and let $\gamma:dgm(B_1)\to dgm(B_2)$ be an optimal matching. Denote by $(x_i,y_i)$ the points of $dgm(B_1)$ and $(x_i',y_i')=\gamma((x_i,y_i))$ their matches in $dgm(B_2)$. Write $\ell_i=y_i-x_i$ and $\ell_i'=y_i'-x_i'$.

Since the bottleneck distance is defined using the $\ell^\infty$ norm, we have
\[
|x_i-x_i'|\leq \delta, \quad |y_i-y_i'|\leq \delta,
\]
and therefore
\[
|\ell_i-\ell_i'|
= |(y_i-x_i)-(y_i'-x_i')|
\leq |y_i-y_i'|+|x_i-x_i'|
\leq 2\delta.
\]

Summing over all $i$ gives
\[
|L^{B_1}-L^{B_2}|
= \left|\sum_i \ell_i - \sum_i \ell_i'\right|
\leq \sum_i |\ell_i-\ell_i'|
\leq 2n\delta.
\]

We now estimate the difference of TSI:
\begin{align*}
|\tsi(B_1)-\tsi(B_2)|
&=\frac{1}{n-1}\left|\sum_i \ell_i^2 - \frac{1}{n}(L^{B_1})^2
-\sum_i (\ell_i')^2 + \frac{1}{n}(L^{B_2})^2\right|\\
&\leq \frac{1}{n-1}\sum_i |\ell_i^2-(\ell_i')^2|
+\frac{1}{n-1}\left|\frac{1}{n}(L^{B_1})^2-\frac{1}{n}(L^{B_2})^2\right|.
\end{align*}

For the first term, we use
\[
|\ell_i^2-(\ell_i')^2|
=|\ell_i-\ell_i'||\ell_i+\ell_i'|
\leq 2\delta\,|\ell_i+\ell_i'|,
\]
so that
\[
\sum_i |\ell_i^2-(\ell_i')^2|
\leq 2\delta \sum_i (\ell_i+\ell_i')
=2\delta\left(L^{B_1}+L^{B_2}\right).
\]

For the second term, we use
\[
|(L^{B_1})^2-(L^{B_2})^2|
=|L^{B_1}-L^{B_2}|\cdot|L^{B_1}+L^{B_2}|
\leq 2n\delta\left(L^{B_1}+L^{B_2}\right).
\]

Combining the estimates, we obtain
\begin{align*}
|\tsi(B_1)-\tsi(B_2)|
&\leq \frac{1}{n-1}\left[
2\delta\left(L^{B_1}+L^{B_2}\right)
+\frac{1}{n}\cdot 2n\delta\left(L^{B_1}+L^{B_2}\right)
\right]\\
&=\frac{1}{n-1}\left[
2\delta\left(L^{B_1}+L^{B_2}\right)
+2\delta\left(L^{B_1}+L^{B_2}\right)
\right]\\
&=\frac{4}{n-1}\delta\left(L^{B_1}+L^{B_2}\right).
\end{align*}
\end{proof}


\subsection{Mean Persistence and a Signal Strength Interpretation}

While the TSI captures the dispersion of persistence lifetimes, it is also natural to consider a complementary quantity that measures their typical scale. This leads to a simple first-moment-type summary, which we interpret as a measure of topological signal strength.

\begin{definition}
Let $B = \{[b_i, d_i)\}_{i=1}^{n^B}$ be a persistence barcode with lifetimes $\ell_i = d_i - b_i$ and total persistence
\[
L^B = \sum_{i=1}^{n^B} \ell_i.
\]
Assume $L^B > 0$. The \emph{Topological Signal Index} (TSigI) of $B$ is defined by
\begin{equation}
\label{eq:tsigi}
\mathrm{TSigI}(B) := \frac{\sum_{i=1}^{n^B} \ell_i^2}{\sum_{i=1}^{n^B} \ell_i}.
\end{equation}
\end{definition}

The quantity $\mathrm{TSigI}(B)$ can be interpreted as a persistence-weighted mean of the lifetimes. Indeed, defining weights
\[
p_i := \frac{\ell_i}{L^B}, \qquad \sum_{i=1}^{n^B} p_i = 1,
\]
we obtain
\[
\mathrm{TSigI}(B) = \sum_{i=1}^{n^B} p_i \, \ell_i.
\]
Thus, $\mathrm{TSigI}(B)$ represents the expected lifetime of a feature sampled proportionally to its persistence.

\begin{remark}\rm
Let $B$ be a barcode with $n^B \geq 1$ and $L^B > 0$. Then:
\begin{enumerate}
    \item $\displaystyle \frac{L^B}{n^B} \leq \mathrm{TSigI}(B) \leq L^B$,
    \item equality on the left holds if and only if all lifetimes are equal,
    \item equality on the right holds if and only if the persistence is fully concentrated in a single bar,
    \item for $c > 0$, $\mathrm{TSigI}(cB) = c\, \mathrm{TSigI}(B)$.
\end{enumerate}
These statements follow directly from elementary inequalities and the arguments used in Theorem~\ref{thm:tsi_extremal}.
\end{remark}

The Topological Signal Index complements the TSI by capturing the \emph{typical scale} of persistence lifetimes, whereas the TSI measures their \emph{dispersion}. The two quantities are directly related through the identity:
\begin{equation}
\label{eq:tsi_tsigi_relation}
\mathrm{TSI}(B)
=
\frac{1}{n^B - 1}
\left(
\sum_{i=1}^{n^B} \ell_i^2 - \frac{(L^B)^2}{n^B}
\right)
=
\frac{L^B}{n^B - 1}
\left(
\mathrm{TSigI}(B) - \frac{L^B}{n^B}
\right).
\end{equation}

\subsubsection{A Hierarchy of Topological Signal Moments}

The relation between \(\tsi(B)\) and \(\text{TSigI}(B)\) suggests that these quantities should not be viewed as isolated descriptors. Rather, they arise naturally from consecutive powers of the persistence lifetimes. This point of view leads to a hierarchy of scalar summaries that interpolate between average persistence and the largest persistence lifetime, resembling the ideas behind the {\it Lehmer mean} \cite{LEHMER1971183}. More precisely, we have:

\begin{definition}
For \(k\geq 1\), the \(k\)-th \emph{Topological Signal Moment} of \(B\) is defined by
\[
M_k(B):=
\frac{\sum_{i=1}^{n}\ell_i^k}
{\sum_{i=1}^{n}\ell_i^{k-1}},
\]
whenever the denominator is nonzero.
\end{definition}

The quantity \(M_k(B)\) compares two consecutive power sums of the lifetime multiset. Equivalently, it may be written as
\[
M_k(B)
=
\sum_{i=1}^{n} w_i^{(k)}\ell_i,
\qquad
w_i^{(k)}
=
\frac{\ell_i^{k-1}}{\sum_{j=1}^{n}\ell_j^{k-1}}.
\]
Thus \(M_k(B)\) is a weighted average of the lifetimes, where the weights depend on the \((k-1)\)-st powers of the lifetimes. As \(k\) increases, longer bars receive progressively greater weight. The first two members of this hierarchy recover familiar quantities. Indeed,

\[
M_1(B)=\frac{L^B}{n} \quad \text{and} \quad M_2(B)=\text{TSigI}(B).
\]

This shows that \(TSigI\) is not an ad hoc companion to the TSI. It is the second member of a natural sequence of lifetime-based descriptors. We now rewrite the identity between \(TSI\) and \(TSigI\) in terms of the moment hierarchy.

\begin{prop}
For every barcode \(B\) with \(n\geq 2\),
\[
\tsi(B)
=
\frac{L^B}{n-1}
\left(
M_2(B)-M_1(B)
\right).
\]
\end{prop}

\begin{proof}
\[
\begin{array}{lclcl}
\tsi(B) & = & \frac{1}{n-1}\left(\sum_{i=1}^{n}\ell_i^2 -
\frac{(L^B)^2}{n}\right)  &\overset{\sum_i\ell_i^2=L^B M_2(B)}{\underset{L^B=n M_1(B)}{=}}& \frac{1}{n-1}\left(L^B M_2(B)-L^B M_1(B)\right)\ =\\
&&&&\\
& = &  \frac{L^B}{n-1}\left(M_2(B)-M_1(B)\right)&&
\end{array}
\]
\end{proof}

Consequently, the TSI measures the discrepancy between ordinary average persistence and persistence-weighted average persistence. Since \(M_2(B)\ge M_1(B)\), with equality if and only if all lifetimes are equal, uniform lifetimes imply \(M_1(B)=M_2(B)\), and therefore the TSI vanishes. As persistence becomes increasingly concentrated among a small number of dominant bars, the weighted average \(M_2(B)\) moves further away from the ordinary average \(M_1(B)\),
reflecting increasing heterogeneity in the barcode and producing larger TSI values.

\smallbreak 

The hierarchy also has a natural monotonicity property. In particular:

\begin{prop}
For every barcode \(B\) with nonnegative lifetimes and at least one positive lifetime,
\[
M_1(B)\leq M_2(B)\leq M_3(B)\leq \cdots .
\]
\end{prop}

\begin{proof}
We prove that \(M_k(B)\leq M_{k+1}(B)\) for every \(k\geq 1\). Let \(S_k(B):=\sum_{i=1}^{n}\ell_i^k\). Then, in terms of power sums, the inequality is equivalent to
\[
\frac{S_k(B)}{S_{k-1}(B)}
\leq
\frac{S_{k+1}(B)}{S_k(B)}.
\]
Since all denominators are positive, this is equivalent to \( S_k(B)^2 \leq S_{k-1}(B)S_{k+1}(B)\). But this follows from the Cauchy--Schwarz inequality applied to the two vectors
\[
\left(\ell_1^{(k-1)/2},\ldots,\ell_n^{(k-1)/2}\right) \quad \text{and} \quad \left(\ell_1^{(k+1)/2},\ldots,\ell_n^{(k+1)/2}\right).
\]
Indeed,
\[
\left(
\sum_{i=1}^{n}
\ell_i^{(k-1)/2}\ell_i^{(k+1)/2}
\right)^2
\leq
\left(
\sum_{i=1}^{n}\ell_i^{k-1}
\right)
\left(
\sum_{i=1}^{n}\ell_i^{k+1}
\right).
\]
The left-hand side is exactly \(S_k(B)^2\), while the right-hand side is \(S_{k-1}(B)S_{k+1}(B)\). Since, \(M_k(B):=\frac{S_k(B)}{S_{k-1}(B)}\), we obtain
\[
M_k(B)\leq M_{k+1}(B),
\]
as required.
\end{proof}

This monotonicity confirms the interpretation of the hierarchy. The first moment describes average persistence. Higher moments increasingly emphasize
larger bars. Thus, the sequence \(M_1(B),M_2(B),M_3(B),\ldots\) moves from global average behavior toward dominant persistent features. The limiting case is the maximal lifetime.

\begin{prop}
Let \(\ell_{\max}:=\max_i \ell_i\). Then
\[
\lim_{k\to\infty} M_k(B)=\ell_{\max}.
\]
\end{prop}

\begin{proof}
Write
\[
M_k(B)=
\frac{\sum_i\ell_i^k}{\sum_i\ell_i^{k-1}}.
\]
Divide numerator and denominator by \(\ell_{\max}^{k-1}\). This gives
\[
M_k(B)
=
\frac{
\sum_i
\ell_i
\left(
\frac{\ell_i}{\ell_{\max}}
\right)^{k-1}
}
{
\sum_i
\left(
\frac{\ell_i}{\ell_{\max}}
\right)^{k-1}
}.
\]
If \(\ell_i<\ell_{\max}\), then
\[
\left(
\frac{\ell_i}{\ell_{\max}}
\right)^{k-1}
\to 0
\]
as \(k\to\infty\). The only terms that remain are those for which
\(\ell_i=\ell_{\max}\). Therefore the numerator converges to \(m\ell_{\max}\), and the denominator converges to \(m\), where \(m\) is the number of bars with maximal lifetime. Hence
\[
\lim_{k\to\infty}M_k(B)=\ell_{\max}.
\]
\end{proof}

The hierarchy further satisfies a natural scaling property. If \(c>0\), then
\[
M_k(cB)
=
cM_k(B).
\]
Indeed,
\[
M_k(cB)
=
\frac{\sum_i(c\ell_i)^k}{\sum_i(c\ell_i)^{k-1}}
=
c
\frac{\sum_i\ell_i^k}{\sum_i\ell_i^{k-1}}
=
cM_k(B).
\]

Hence every \(M_k\) is homogeneous of degree one. In this sense, the topological signal moments behave like scale descriptors, whereas the TSI, being variance-based, is homogeneous of degree two.

The hierarchy therefore provides a gradual transition from average persistence to dominant persistence features. Lower-order moments capture global characteristics of the barcode, whereas higher-order moments progressively emphasize the longest-lived topological structures.

\medskip

\paragraph{{\bf A Two-Dimensional Summary of Persistence Barcodes}}

The preceding relation
\[
\tsi(B)=\frac{L^B}{n-1} \left(M_2(B)-M_1(B)\right)
\]
shows that the TSI quantifies the discrepancy between ordinary average persistence and persistence-weighted average persistence. This naturally
suggests viewing \(TSI\) and \(TSigI\) as complementary descriptors. We therefore propose to view the pair
\((\mathrm{TSigI}(B),\mathrm{TSI}(B))\) as a joint summary capturing first- and second-order structure of persistence lifetimes (for an illustration see Figure~\ref{fig:noisy_circles_total}). Specifically, the pair \((\mathrm{TSigI}(B),\mathrm{TSI}(B))\) provides a two-dimensional description of the persistence barcode, encoding complementary aspects of its structure:
\begin{itemize}
    \item[i.] $\mathrm{TSigI}$ quantifies the typical strength of persistence lifetimes (signal strength),
    \item[ii.] $\mathrm{TSI}$ quantifies the dispersion of  lifetimes (structural variability).
\end{itemize}

Together, these quantities encode both the magnitude and the heterogeneity of topological features. In particular:
\begin{itemize}
    \item[-] high $\mathrm{TSigI}$ and low $\mathrm{TSI}$ indicate concentration of persistence around a dominant scale,
    \item[-] low $\mathrm{TSigI}$ and low $\mathrm{TSI}$ indicate low total persistence and absence of dominant features,
    \item[-] high $\mathrm{TSI}$ indicates significant variability in persistence lifetimes, corresponding to multiscale structure or structural transitions.
\end{itemize}

The quantity $\mathrm{TSigI}(B)$ is closely related to energy-type functionals appearing in persistence landscapes \cite{bubenik_statistical_2015}, where integrals of the form $\int \lambda(t)^2 \, dt$ encode second-order persistence information. In this sense, $\mathrm{TSigI}$ may be viewed as a scalar, lifetime-based analogue of such functional summaries.

\section{Relation to Persistence Entropy}\label{Sec:Relation to Persistence Entropy}
Persistent entropy is a widely used scalar summary of persistence barcodes, designed to capture the distribution of persistence lifetimes in a scale-invariant way \cite{rucco_characterisation_2016,atienza_persistent_2019,atienza_stability_2020}. In this section, we compare the Topological Stability Index with persistent entropy, and show that a normalized version of the TSI admits an exact algebraic relation to the Rényi entropy of order two. This connection clarifies the mathematical meaning of the TSI and places it within the broader family of entropy-based persistence summaries.


\subsection{Persistent Entropy and Conceptual Comparison}

\begin{definition}[\cite{atienza_persistent_2019}]
Let $B=\{[b_i,d_i)\}_{i=1}^{n^B}$ be a barcode with total persistence $L^B>0$, and let $\ell_i=d_i-b_i$ denote the corresponding bar lengths. Define
\[
L^B:=\sum_{i=1}^{n^B}\ell_i,
\qquad
p_i:=\frac{\ell_i}{L^B}.
\]
The \emph{persistent entropy} of $B$ is
\begin{equation}
E(B):=-\sum_{i=1}^{n^B} p_i \log p_i.
\end{equation}
\end{definition}

Both TSI and persistent entropy summarize the distribution of persistence lifetimes, but they emphasize different aspects of that distribution. The TSI measures the dispersion of lifetimes around their mean, and is therefore sensitive to absolute scale and heterogeneity. In contrast, persistent entropy depends only on the normalized weights $p_i$, and thus measures the relative distribution of persistence mass. Consequently, the TSI is not scale invariant, while persistent entropy is invariant under uniform scaling of all lifetimes.

\begin{prop}
There exist barcodes $B_1$ and $B_2$ such that
\[
E(B_1)=E(B_2)
\quad \text{but} \quad
\tsi(B_1)\neq \tsi(B_2).
\]
\end{prop}

\begin{proof}
Consider two barcodes whose lifetimes are given by
\[
B_1=(1,2), \qquad B_2=(2,4).
\]
They induce the same normalized weights,
\[
\left(\frac13,\frac23\right),
\]
and hence
\[
E(B_1)=E(B_2).
\]
However, since $B_2=2B_1$, Proposition~\ref{prop:scalar_mult_tsi} gives
\[
\tsi(B_2)=4\,\tsi(B_1).
\]
Since $\tsi(B_1)>0$, it follows that
\[
\tsi(B_1)\neq \tsi(B_2).
\]
\end{proof}

The example in the proof is depicted in Figure~\ref{fig:entropy_same_scale_diff_tsi}.

\begin{figure}[htbp]
\centering
\begin{subfigure}[t]{0.49\textwidth}
    \centering
    \includegraphics[height=2in]{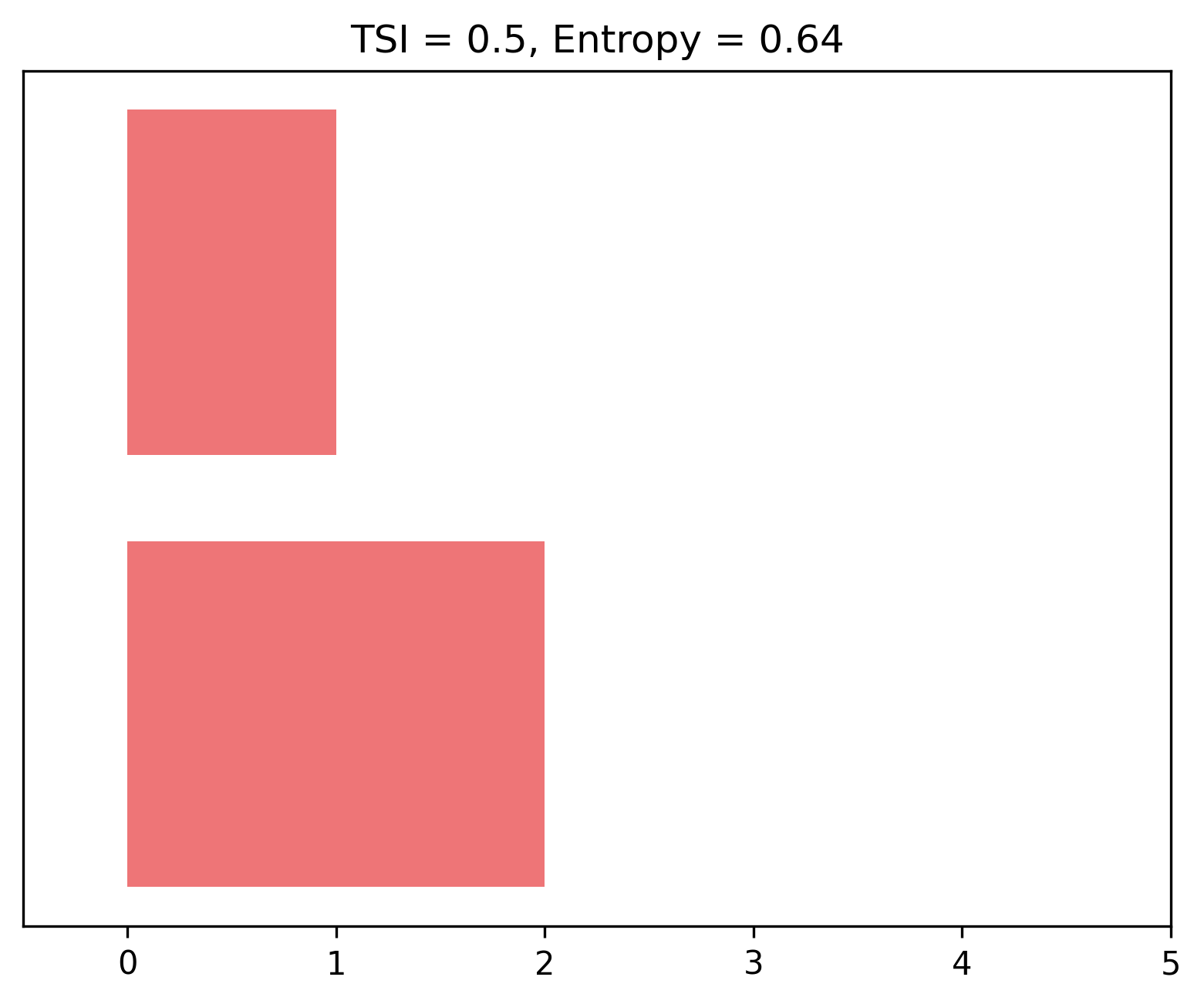}
    \caption{A barcode with lifetimes $(1,2)$.}
\end{subfigure}
~
\begin{subfigure}[t]{0.49\textwidth}
    \centering
    \includegraphics[height=2in]{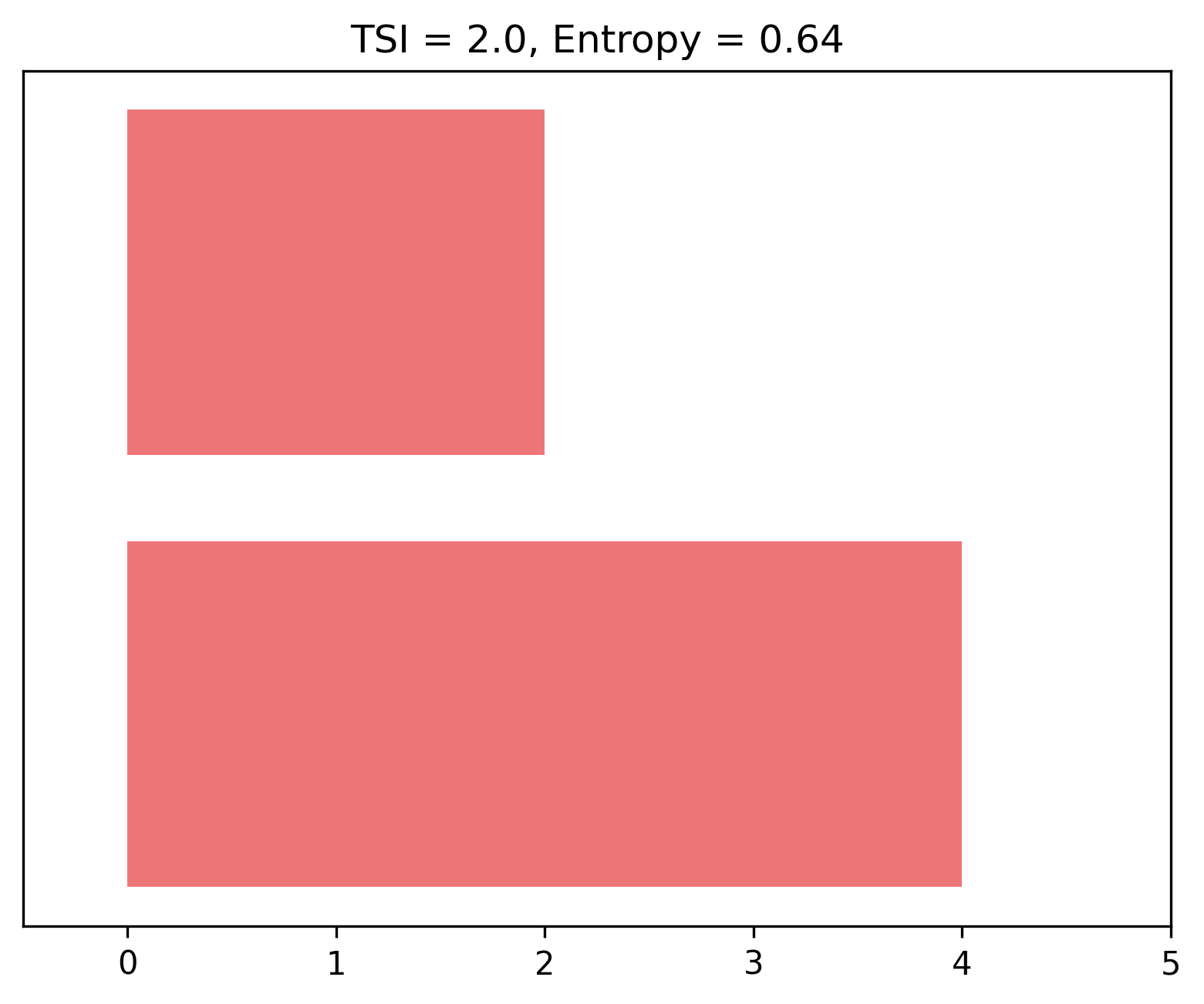}    \caption{The same barcode scaled by a factor of two.}
\end{subfigure}
\caption{Two barcodes with the same normalized lifetime distribution $(1/3,2/3)$ and therefore the same persistent entropy, but different total scale. The right barcode is obtained by uniformly scaling the lifetimes of the left by a factor of two. Persistent entropy is invariant under such scaling, whereas the TSI scales quadratically.}
\label{fig:entropy_same_scale_diff_tsi}
\end{figure}

A second important distinction concerns perturbations that change the number of bars. As shown in Section~\ref{Sec:The Topological Stability Index}, the TSI is sensitive to the insertion of new bars, even when the inserted bar has arbitrarily small length. Persistent entropy behaves differently: it is continuous under this type of perturbation \cite{atienza_stability_2020}.

Indeed, suppose a new bar of length $\ell$ is added to $B$. Then the normalized weights become
\[
\tilde p_i=\frac{\ell_i}{L^B+\ell},
\qquad
p=\frac{\ell}{L^B+\ell}.
\]
As $\ell\to 0$, one has $\tilde p_i\to p_i$ and $p\to 0$. Therefore,
\[
E(B\cup\{[b,b+\ell)\})
=
-\sum_{i=1}^{n^B}\tilde p_i\log \tilde p_i - p\log p
\longrightarrow
-\sum_{i=1}^{n^B}p_i\log p_i
=
E(B),
\]
since $x\log x\to 0$ as $x\to 0^+$.

Thus, persistent entropy provides a scale-invariant and insertion-continuous summary, whereas the TSI retains absolute information about the spread of lifetimes and is correspondingly more sensitive to local modifications of the barcode.


\subsection{Normalized TSI and Rényi Entropy}

To better relate TSI to entropy-based summaries, we introduce a normalized version.

\begin{definition}
Let $B$ be a barcode with $n^B\geq 1$ bars and total persistence $L^B>0$. The \emph{normalized Topological Stability Index} is defined by
\begin{equation}
cv\tsi(B):=\frac{\tsi(B)}{\mathbb{E}[\mathcal L_B]^2},
\end{equation}
where $\mathbb{E}[\mathcal L_B]=L^B/n^B$ is the mean bar length.
\end{definition}

Using the explicit formula for the TSI, we immediately obtain the following expression.

\begin{prop}\label{prop:cvtsi_explicit}
For any barcode $B$ with $n^B\geq 2$,
\begin{equation}
cv\tsi(B)
=
\frac{1}{n^B-1}\sum_{i=1}^{n^B}\left(n^B\frac{\ell_i}{L^B}-1\right)^2.
\end{equation}
Equivalently, in terms of the normalized weights $p_i=\ell_i/L^B$,
\begin{equation}
cv\tsi(B)
=
\frac{1}{n^B-1}\sum_{i=1}^{n^B}(n^Bp_i-1)^2.
\end{equation}
\end{prop}

\begin{proof}
By definition,
\[
cv\tsi(B)
=
\frac{\tsi(B)}{(L^B/n^B)^2}.
\]
Using the explicit formula for $\tsi(B)$ from Section~\ref{Sec:The Topological Stability Index},
\[
\tsi(B)
=
\frac{1}{n^B-1}\sum_{i=1}^{n^B}\left(\ell_i-\frac{L^B}{n^B}\right)^2,
\]
we obtain
\begin{align*}
cv\tsi(B)
&=
\frac{1}{n^B-1}
\sum_{i=1}^{n^B}
\frac{\left(\ell_i-\frac{L^B}{n^B}\right)^2}{(L^B/n^B)^2}\\
&=
\frac{1}{n^B-1}
\sum_{i=1}^{n^B}
\left(\frac{n^B\ell_i}{L^B}-1\right)^2,
\end{align*}
which proves the result.
\end{proof}

\begin{theorem}\label{thm:cvtsi_bounds}
Let $B$ be a barcode with $n^B=n\geq 2$ and $L^B>0$. Then
\[
0\leq cv\tsi(B)\leq n.
\]
Moreover:
\begin{enumerate}
    \item $cv\tsi(B)=0$ if and only if all bar lengths are equal.
    \item $cv\tsi(B)=n$ if and only if, up to permutation, one bar has length $L^B$ and the remaining $n-1$ bars have length $0$.
\end{enumerate}
\end{theorem}

\begin{proof}
By definition,
\[
cv\tsi(B)=\frac{\tsi(B)}{(L^B/n)^2}.
\]
Applying Theorem~\ref{thm:tsi_extremal}, we get
\[
0\leq \tsi(B)\leq \frac{(L^B)^2}{n}.
\]
Dividing through by $(L^B/n)^2$ yields
\[
0\leq cv\tsi(B)\leq \frac{(L^B)^2/n}{(L^B)^2/n^2}=n.
\]
The equality cases follow immediately from those of Theorem~\ref{thm:tsi_extremal}.
\end{proof}

\begin{remark}\rm
For a barcode $B$ with $n$ bars, the normalized quantity $\frac{cv\tsi(B)}{n}$ takes values in $[0,1]$ and admits a natural interpretation in terms of the distribution of persistence lifetimes:
\begin{itemize}
    \item[-] Values close to $0$ correspond to nearly uniform lifetimes. 
    \item[-] Values close to $1$ correspond to highly concentrated persistence.
    \item[-] Intermediate values (around $1/2$) indicate a moderate level of heterogeneity.
\end{itemize}
Thus, $\frac{cv\tsi(B)}{n}$ can be viewed as normalized measure of concentration for the lifetime distribution. 
\end{remark}

\begin{corollary}
For fixed $n$, the normalized TSI is smallest when the barcode lifetimes are perfectly uniform and largest when the persistence mass is maximally concentrated. Equivalently, $cv\tsi(B)$ is a concentration index for the normalized lifetime distribution $(p_1,\dots,p_n)$.
\end{corollary}

\begin{remark}\rm 
    The preceding theorem and corollary demonstrate that the normalized TSI captures the dispersion of a barcode independently of the total bar length. Nevertheless, its scale remains dependent on the number of bars, $n$. Consequently, caution is required when using $cv\tsi$ to compare barcodes containing significantly different numbers of bars, as their scales may not be directly comparable. In such cases, the quantity $\frac{cv\tsi}{n}$ serves as a more suitable metric because it is strictly bounded within the interval [0, 1]. Specifically, a value of 0 is achieved if and only if all bars are of uniform length, whereas a value of 1 occurs exclusively when a single bar accounts for the entirety of the total persistence.
\end{remark}

The normalized quantity $cv\tsi(B)$ depends only on the proportions $p_i=\ell_i/L^B$, and is therefore scale invariant. More importantly, it admits an exact algebraic relation to the collision probability $\sum_i p_i^2$.

\begin{prop}\label{prop:collision_identity}
Let $B$ be a barcode with $n^B\geq 2$ and $L^B>0$, and let $p_i=\ell_i/L^B$. Then
\begin{equation}
\sum_{i=1}^{n^B} p_i^2
=
\frac{1}{n^B}
+
\frac{n^B-1}{(n^B)^2}cv\tsi(B).
\end{equation}
\end{prop}

\begin{proof}
Starting from Proposition~\ref{prop:cvtsi_explicit},
\begin{align*}
cv\tsi(B)
&=
\frac{1}{n^B-1}\sum_{i=1}^{n^B}(n^Bp_i-1)^2\\
&=
\frac{1}{n^B-1}\sum_{i=1}^{n^B}\left((n^B)^2p_i^2-2n^Bp_i+1\right).
\end{align*}
Since $\sum_i p_i=1$, this becomes
\begin{align*}
cv\tsi(B)
&=
\frac{1}{n^B-1}
\left((n^B)^2\sum_{i=1}^{n^B}p_i^2-2n^B+n^B\right)\\
&=
\frac{1}{n^B-1}
\left((n^B)^2\sum_{i=1}^{n^B}p_i^2-n^B\right).
\end{align*}
Rearranging yields
\[
(n^B)^2\sum_{i=1}^{n^B}p_i^2
=
n^B+(n^B-1)cv\tsi(B),
\]
and hence
\[
\sum_{i=1}^{n^B} p_i^2
=
\frac{1}{n^B}
+
\frac{n^B-1}{(n^B)^2}cv\tsi(B).
\]
\end{proof}

\begin{prop}\label{prop:cvtsi_l2_uniform}
Let $B$ be a barcode with $n^B\geq 2$ and $L^B>0$. Let $u=(1/n^B,\dots,1/n^B)$ and let $p=(p_1,\dots,p_{n^B})$ be the normalized lifetime distribution. Then
\[
cv\tsi(B)=\frac{(n^B)^2}{n^B-1}\|p-u\|_2^2.
\]
In particular, $cv\tsi(B)$ is a rescaled squared Euclidean distance from the uniform distribution.
\end{prop}

\begin{proof}
Since $u_i=1/n^B$, we have
\[
\|p-u\|_2^2
=
\sum_{i=1}^{n^B}\left(p_i-\frac{1}{n^B}\right)^2
=
\sum_{i=1}^{n^B}p_i^2-\frac{2}{n^B}\sum_{i=1}^{n^B}p_i+\frac{1}{n^B}.
\]
Using $\sum_i p_i=1$, this becomes
\[
\|p-u\|_2^2
=
\sum_{i=1}^{n^B}p_i^2-\frac{1}{n^B}.
\]
By Proposition~\ref{prop:collision_identity},
\[
cv\tsi(B)
=
\frac{(n^B)^2}{n^B-1}\left(\sum_{i=1}^{n^B}p_i^2-\frac{1}{n^B}\right)
=
\frac{(n^B)^2}{n^B-1}\|p-u\|_2^2.
\]
This proves the claim.
\end{proof}

We now recall the Rényi entropy of order $\alpha$ \cite{renyi1961}.

\begin{definition}
Let $(p_1,\dots,p_n)$ be a probability vector and let $\alpha>0$, $\alpha\neq 1$. The \emph{Rényi entropy of order $\alpha$} is defined by
\begin{equation}
H_\alpha(p_1,\dots,p_n)
:=
\frac{1}{1-\alpha}\log\left(\sum_{i=1}^n p_i^\alpha\right).
\end{equation}
For a barcode $B$, we write $H_\alpha(B)$ for the Rényi entropy of the associated weights $p_i=\ell_i/L^B$.
\end{definition}

In the case $\alpha=2$, this becomes
\[
H_2(B)=-\log\left(\sum_{i=1}^{n^B}p_i^2\right).
\]
Combining this with Proposition~\ref{prop:collision_identity}, we obtain the exact link between $cv\tsi$ and Rényi entropy.

\begin{corollary}\label{cor:renyi_relation}
For any barcode $B$ with $n^B\geq 2$ and $L^B>0$,
\begin{equation}
H_2(B)
=
-\log\left(
\frac{1}{n^B}
+
\frac{n^B-1}{(n^B)^2}cv\tsi(B)
\right).
\end{equation}
Equivalently,
\begin{equation}
cv\tsi(B)
=
\frac{(n^B)^2}{n^B-1}
\left(
e^{-H_2(B)}-\frac{1}{n^B}
\right).
\end{equation}
\end{corollary}

Corollary~\ref{cor:renyi_relation} shows that the normalized TSI is not merely related to the Rényi entropy of order two, but it is in fact a monotone reparametrization of it. In particular, $cv\tsi(B)$ and $H_2(B)$ contain exactly the same information about the distribution of normalized persistence lifetimes. This relationship is illustrated in Figure~\ref{fig:cvtsi_h2_extremal}. However, the two quantities emphasize different aspects of this information. The Rényi entropy is logarithmic (and therefore it compresses differences near highly concentrated configurations), while $cv\tsi(B)$ depends linearly on $e^{-H_2(B)}$ (and thus provides a more sensitive measure). This distinction is illustrated in Figure~\ref{fig:sampled_circles_alpha}(b).

\begin{figure}[htbp]
\centering
\begin{subfigure}[t]{0.49\textwidth}
    \centering
    \includegraphics[height=2in]{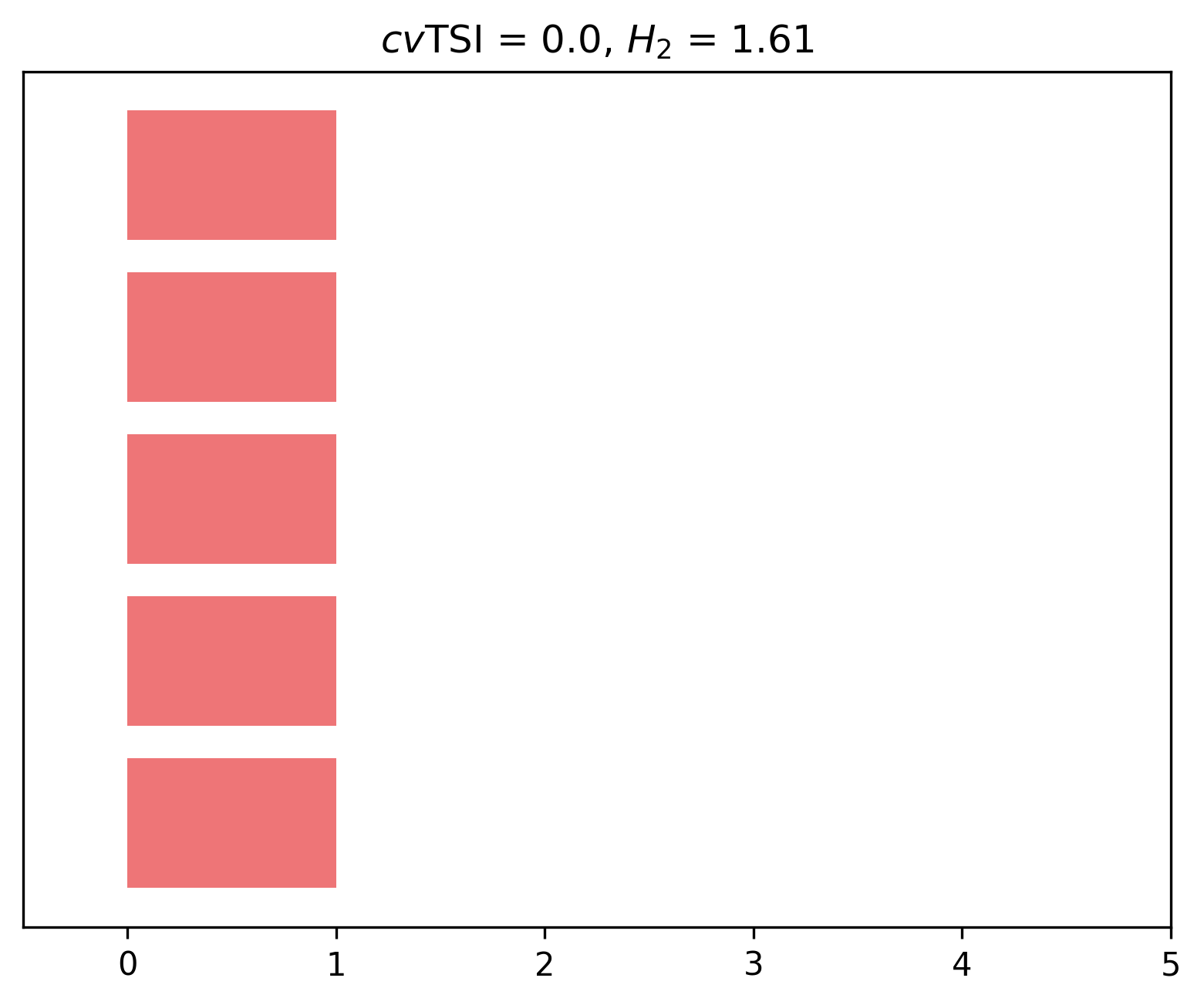}
    \caption{Uniform lifetime distribution.}
\end{subfigure}
~
\begin{subfigure}[t]{0.49\textwidth}
    \centering
    \includegraphics[height=2in]{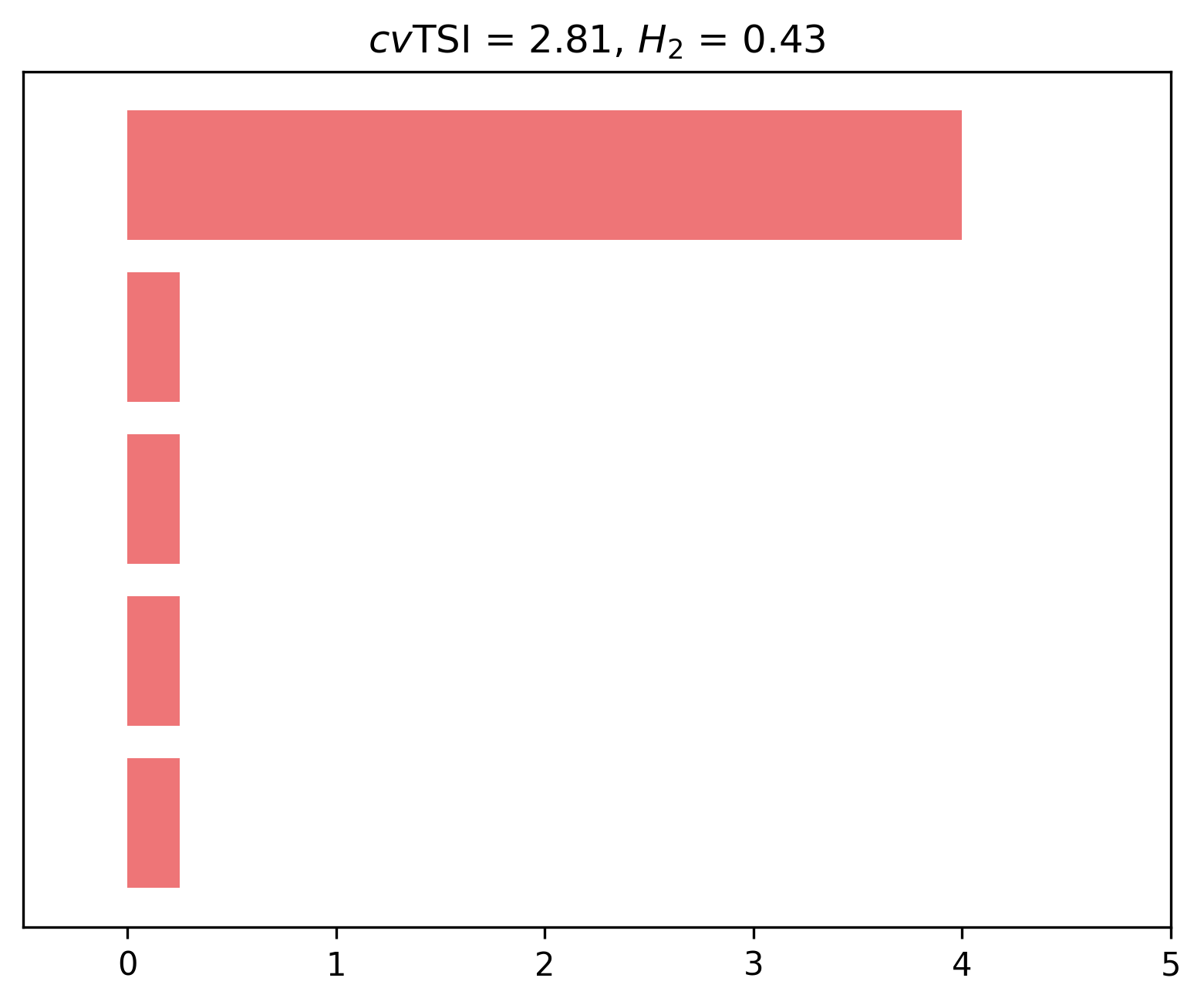}
    \caption{Concentrated lifetime distribution.}
\end{subfigure}
\caption{Extremal behavior of the normalized TSI and Rényi entropy of order two. Uniform lifetime distributions minimize $cv\tsi$ and maximize $H_2$, whereas maximally concentrated lifetime distributions maximize $cv\tsi$ and minimize $H_2$.}
\label{fig:cvtsi_h2_extremal}
\end{figure}

\begin{prop}\label{prop:entropy_expansion}
Let $p=(p_1,\dots,p_{n^B})$ be a probability vector of the form
\[
p_i=\frac{1}{n^B}+\varepsilon_i,
\qquad
\sum_i \varepsilon_i=0,
\]
with $\|\varepsilon\|$ sufficiently small. Then
\[
E(B)=\log(n^B)-\frac{n^B-1}{2n^B}cv\tsi(B)+O(\|\varepsilon\|^3),
\]
with respect to any fixed norm on $\mathbb{R}^n$.
\end{prop}

\begin{proof}
Let $n=n^B$ and write
\[
p_i=\frac{1}{n}+\varepsilon_i,
\qquad
\sum_{i=1}^n \varepsilon_i=0,
\]
with $\|\varepsilon\|$ sufficiently small. We expand the function
\[
f(x):=-x\log x
\]
around the point $x_0=\frac{1}{n}$.

Since
\[
f'(x)=-(\log x+1),
\qquad
f''(x)=-\frac{1}{x},
\qquad
f'''(x)=\frac{1}{x^2},
\]
Taylor's formula gives, for each $i$,
\[
f\!\left(\frac{1}{n}+\varepsilon_i\right)
=
f\!\left(\frac{1}{n}\right)
+
f'\!\left(\frac{1}{n}\right)\varepsilon_i
+
\frac12 f''\!\left(\frac{1}{n}\right)\varepsilon_i^2
+
O(|\varepsilon_i|^3).
\]
Now
\[
f\!\left(\frac{1}{n}\right)=\frac{1}{n}\log n,
\qquad
f'\!\left(\frac{1}{n}\right)=\log n-1,
\qquad
f''\!\left(\frac{1}{n}\right)=-n.
\]
Hence
\[
-p_i\log p_i
=
\frac{1}{n}\log n
+
(\log n-1)\varepsilon_i
-\frac{n}{2}\varepsilon_i^2
+
O(|\varepsilon_i|^3).
\]

Summing over $i=1,\dots,n$, we obtain
\begin{align*}
E(B)
&=
\sum_{i=1}^n \left(-p_i\log p_i\right)\\
&=
\sum_{i=1}^n
\left(
\frac{1}{n}\log n
+
(\log n-1)\varepsilon_i
-\frac{n}{2}\varepsilon_i^2
+
O(|\varepsilon_i|^3)
\right)\\
&=
\log n
+
(\log n-1)\sum_{i=1}^n \varepsilon_i
-\frac{n}{2}\sum_{i=1}^n \varepsilon_i^2
+
O(\|\varepsilon\|^3).
\end{align*}
Since $\sum_i \varepsilon_i=0$, this simplifies to
\[
E(B)=\log n-\frac{n}{2}\sum_{i=1}^n \varepsilon_i^2+O(\|\varepsilon\|^3).
\]

It remains to express the quadratic term in terms of $cv\tsi(B)$. By Proposition~\ref{prop:cvtsi_l2_uniform},
\[
cv\tsi(B)=\frac{n^2}{n-1}\|p-u\|_2^2,
\]
where $u=(1/n,\dots,1/n)$. But
\[
\|p-u\|_2^2=\sum_{i=1}^n \left(p_i-\frac{1}{n}\right)^2
=\sum_{i=1}^n \varepsilon_i^2.
\]
Therefore,
\[
\sum_{i=1}^n \varepsilon_i^2=\frac{n-1}{n^2}cv\tsi(B).
\]
Substituting this into the previous expression yields
\[
E(B)
=
\log n
-\frac{n}{2}\cdot \frac{n-1}{n^2}cv\tsi(B)
+
O(\|\varepsilon\|^3),
\]
that is,
\[
E(B)=\log(n^B)-\frac{n^B-1}{2n^B}cv\tsi(B)+O(\|\varepsilon\|^3).
\]
This completes the proof.
\end{proof}

\begin{remark}\rm
This expansion shows that near the uniform distribution, the normalized TSI captures the leading quadratic deviation of persistent entropy from its maximal value $\log(n^B)$.
\end{remark}

The previous result gives a precise interpretation of the normalized TSI. Since $\sum_i p_i^2$ is the collision probability of the distribution $(p_i)$, the quantity $cv\tsi(B)$ measures concentration of persistence mass. Small values correspond to a more even distribution of lifetimes, while large values indicate that a relatively small number of bars carry most of the total persistence.

Finally, the usual Shannon-Rényi inequality \cite{renyi1961} yields a direct comparison with persistent entropy.

\begin{prop}\label{prop:entropy_bound}
For any barcode $B$ with $n^B\geq 2$ and $L^B>0$,
\begin{equation}
E(B)\geq H_2(B).
\end{equation}
Equivalently,
\begin{equation}
E(B)\geq -\log\left(
\frac{1}{n^B}
+
\frac{n^B-1}{(n^B)^2}cv\tsi(B)
\right).
\end{equation}
\end{prop}

\begin{proof}
It is a standard fact that Rényi entropy is non-increasing in the parameter $\alpha$. Since Shannon entropy is recovered in the limit $\alpha\to 1$, one has
\[
E(B)=H_1(B)\geq H_2(B).
\]
Substituting the formula from Corollary~\ref{cor:renyi_relation} gives the second inequality.
\end{proof}


\subsection{Perturbation Behavior of the Normalized TSI}

Although $cv\tsi$ is scale invariant, it is still sensitive to perturbations that change the number of bars. Nevertheless, its behavior under insertion of a new bar can be described explicitly.

\begin{lemma}\label{lemma:cvtsi_added_bar}
Let $B$ be a barcode with $n^B\geq 2$ bars, and let $\bar{\ell}_B=L^B/n^B$ be the mean bar length with $L^B>0$. Suppose a new bar of length $\ell$ is added, and define
\[
r:=\frac{\ell-\bar{\ell}_B}{\bar{\ell}_B}.
\]
Then
\begin{equation}
cv\tsi(B\cup\{[b,b+\ell)\})
=
\frac{(n^B+1)^2(n^B-1)}{n^B(n^B+1+r)^2}cv\tsi(B)
+
(n^B+1)\left(\frac{r}{n^B+1+r}\right)^2.
\end{equation}
\end{lemma}

\begin{proof}
Let $B^+:=B\cup\{[b,b+\ell)\}$ and write $n=n^B$ and $\bar\ell=\bar{\ell}_B$. By definition,
\[
\tsi(B)=cv\tsi(B)\,\bar\ell^2.
\]
From Lemma~\ref{lemma:1_changed_bar}, we know that
\[
\tsi(B^+)
=
\frac{n-1}{n}\tsi(B)
+
\frac{1}{n+1}(\ell-\bar\ell)^2.
\]
Moreover, the new mean length is
\[
\bar\ell^+
=
\frac{L^B+\ell}{n+1}
=
\bar\ell+\frac{\ell-\bar\ell}{n+1}.
\]
Since $\ell-\bar\ell=r\bar\ell$, this becomes
\[
\bar\ell^+
=
\bar\ell\left(1+\frac{r}{n+1}\right).
\]
Therefore,
\begin{align*}
cv\tsi(B^+)
&=
\frac{\tsi(B^+)}{(\bar\ell^+)^2}\\
&=
\frac{\frac{n-1}{n}cv\tsi(B)\bar\ell^2+\frac{1}{n+1}r^2\bar\ell^2}
{\bar\ell^2\left(1+\frac{r}{n+1}\right)^2}\\
&=
\frac{\frac{n-1}{n}cv\tsi(B)+\frac{1}{n+1}r^2}
{\left(1+\frac{r}{n+1}\right)^2}.
\end{align*}
Multiplying numerator and denominator by $(n+1)^2$ gives
\[
cv\tsi(B^+)
=
\frac{(n+1)^2(n-1)}{n(n+1+r)^2}cv\tsi(B)
+
(n+1)\left(\frac{r}{n+1+r}\right)^2,
\]
as claimed.
\end{proof}

As a consequence, if $\ell\to 0$, then $r\to -1$, and hence
\begin{equation}
cv\tsi(B\cup\{[b,b+\ell)\})
\longrightarrow
\frac{(n^B+1)^2(n^B-1)}{(n^B)^3}cv\tsi(B)
+
\frac{n^B+1}{(n^B)^2}.
\end{equation}
Thus $cv\tsi$ is not continuous under insertion of arbitrarily short bars either. However, unlike the unnormalized TSI, the limiting discrepancy depends only on the number of bars and not on the total scale of the barcode.

Under fixed cardinality, one again obtains a quantitative stability estimate.

\begin{lemma}\label{lemma:cvtsi_stability}
Let $B_1$ and $B_2$ be barcodes with the same number of bars $n\geq 2$ and with $L^{B_1},L^{B_2}>0$. Then
\begin{equation}
|cv\tsi(B_1)-cv\tsi(B_2)|
\leq
\frac{8n^2}{(n-1)\min\{\bar{\ell}_{B_1},\bar{\ell}_{B_2}\}}\,d_B(B_1,B_2),
\end{equation}
where $\bar{\ell}_{B_i}=L^{B_i}/n$ denotes the mean bar length of $B_i$.
\end{lemma}

\begin{proof}
By Proposition~\ref{prop:cvtsi_explicit},
\begin{align*}
|cv\tsi(B_1)-cv\tsi(B_2)|
&=
\frac{n^2}{n-1}
\left|
\sum_i\left(\frac{\ell_i}{L^{B_1}}\right)^2
-
\sum_i\left(\frac{\ell_i'}{L^{B_2}}\right)^2
\right|\\
&\leq
\frac{n^2}{n-1}
\sum_i
\left|
\left(\frac{\ell_i}{L^{B_1}}\right)^2
-
\left(\frac{\ell_i'}{L^{B_2}}\right)^2
\right|\\
&=
\frac{n^2}{n-1}
\sum_i
\left|
\frac{\ell_i}{L^{B_1}}-\frac{\ell_i'}{L^{B_2}}
\right|
\left(
\frac{\ell_i}{L^{B_1}}+\frac{\ell_i'}{L^{B_2}}
\right).
\end{align*}
Using the normalized stability estimate from \cite{atienza_persistent_2019}, one has
\[
\left|
\frac{\ell_i}{L^{B_1}}-\frac{\ell_i'}{L^{B_2}}
\right|
\leq
\frac{4n}{\min\{L^{B_1},L^{B_2}\}}\,d_B(B_1,B_2).
\]
Therefore,
\begin{align*}
|cv\tsi(B_1)-cv\tsi(B_2)|
&\leq
\frac{4n^3}{(n-1)\min\{L^{B_1},L^{B_2}\}}\,d_B(B_1,B_2)
\sum_i
\left(
\frac{\ell_i}{L^{B_1}}+\frac{\ell_i'}{L^{B_2}}
\right)\\
&=
\frac{4n^3}{(n-1)\min\{L^{B_1},L^{B_2}\}}\,d_B(B_1,B_2)\cdot 2\\
&=
\frac{8n^3}{(n-1)\min\{L^{B_1},L^{B_2}\}}\,d_B(B_1,B_2).
\end{align*}
Since $\min\{L^{B_1},L^{B_2}\}=n\min\{\bar{\ell}_{B_1},\bar{\ell}_{B_2}\}$, this becomes
\[
|cv\tsi(B_1)-cv\tsi(B_2)|
\leq
\frac{8n^2}{(n-1)\min\{\bar{\ell}_{B_1},\bar{\ell}_{B_2}\}}\,d_B(B_1,B_2),
\]
as claimed.
\end{proof}

\begin{remark}\rm 
The role of $cv\tsi$ is conceptually different from that of the original TSI. The TSI captures absolute heterogeneity in persistence lifetimes, while $cv\tsi$ removes scale and measures concentration of persistence mass. Through its exact relation to $H_2$, the normalized TSI should therefore be viewed as an information-theoretic counterpart of the variance-based summary introduced in Section~\ref{Sec:The Topological Stability Index}.
\end{remark}



\subsection{Relation to persistence curves}

Persistence curves provide a general framework for transforming persistence diagrams into functional summaries \cite{chung_persistence_2022}. In that framework, one associates to a persistence diagram a real-valued function of the filtration parameter, typically by aggregating contributions from the features that are alive at each scale.

Although the TSI is not itself a persistence curve in this sense, it may be viewed as a scalar, time-collapsed summary built from the global distribution of persistence lifetimes. More precisely, the TSI aggregates squared deviations of bar lengths from their mean, and therefore plays a role analogous to a global second-moment functional on the barcode.

This viewpoint clarifies the contrast between scalar and functional summaries in topological data analysis. Persistence curves retain scale-by-scale information and can in many cases exhibit stronger stability under local perturbations, whereas the TSI compresses the entire barcode into a single number that measures overall lifetime dispersion. This greater compression makes the TSI simpler and more interpretable, but also more sensitive to certain global modifications, such as the insertion of new bars, as shown in Section~\ref{Sec:The Topological Stability Index}.

\begin{remark}\rm
The persistence-curve framework suggests a natural functional analogue of the TSI: one could define a time-dependent summary by aggregating squared deviations of the lifetimes of features that are alive at each filtration value. Such a construction would retain temporal information while capturing dispersion, and may provide a more stable alternative to the scalar TSI. We leave this direction for future work.
\end{remark}


\section{Numerical Experiments and Applications}\label{Sec:Numerical Experiments and Applications}
In this section, we investigate the empirical behavior of the Topological Stability Index (TSI) and compare it with persistent entropy and its normalized counterpart $cv\tsi$. The experiments are designed to validate the theoretical properties established in \S~\ref{Sec:The Topological Stability Index} and \S~\ref{Sec:Relation to Persistence Entropy} and to highlight the practical differences between dispersion-based and entropy-based summaries.


\subsection{Geometric configurations: circle models}

We begin with controlled geometric examples where the persistence structure is well understood. For this we use the Rips complex as it is easiest to interpret and we look only at dimension 1 features.

\medskip

\paragraph{{\bf Disjoint circles:}}
We consider two disjoint circles of radii $1$ and $\tfrac{1}{4}$, sampled with $n$ equidistant points as illustrated in Figure~\ref{fig:disjoint_circles_total}. In the ideal continuous setting, the persistence diagram contains two dominant $1$-dimensional features whose lifetimes are proportional to the radii, yielding a fixed ratio of persistence lengths.

The death of each loop is determined by the formation of the first triangle that contains the center of the circle. In the continuous limit, this corresponds to the inscribed equilateral triangle. Its side length, and consequently the death time of the $1$-dimensional features, can be computed from the radius $r$ as $s=2r\cos(\pi/6)$.

In the discrete setting with finitely many points, the birth and death times of these features are perturbed. The birth time is determined by the distance between adjacent points, which is inversely proportional to the number of points. The death time depends on the closest available approximation of the theoretical inscribed equilateral triangle. Consequently, we observe features with persistence intervals of the form
\[
    \left[2r_i\sin(\frac{\pi}{n_i}),\,  2r_i\cos(\pi/6)+\varepsilon_{n_i}\right),
\]
where $r_i$ is the radius of circle $i$, $n_i$ is the number of points sampled from circle $i$, and $\varepsilon_{n_i}$ is a decreasing error term arising from the discrete approximation of the equilateral triangle. Note that $\varepsilon_{3n_i}=0$, because in such cases, the points can perfectly form an inscribed equilateral triangle. Moreover, by sampling four times as many points from the larger circle, by ignoring the error, and by using the first order approximation of the sine function, one obtains the following barcode: 
\[\left\{\left[\frac{5\pi}{2n}, 2\cos(\pi/6)\right), \left[\frac{5\pi}{2n}, \frac{1}{2}\cos(\pi/6)\right)\right\}.
\]

The birth times of the bars in this barcode are therefore determined by $n$; however, the death times remain largely unaffected. Consequently, comparing varying values of $n$ is analogous to applying a uniform translation to all bars. Hence, by Proposition~\ref{prop:length_translate_tsi}, the TSI remains invariant. Figure~\ref{fig:disjoint_circles_total} illustrates that the TSI converges almost immediately to its asymptotic value; only the decreasing approximation error is visible. In contrast, persistent entropy and $cv\tsi$ exhibit slower convergence. Both metrics are highly dependent on the birth times converging to $0$, which is reflected in their respective increasing and decreasing convergence behaviors.

\begin{figure}[H]
    \centering
    \begin{subfigure}[t]{0.3\textwidth}
        \centering
        \includegraphics[height=1.5in]{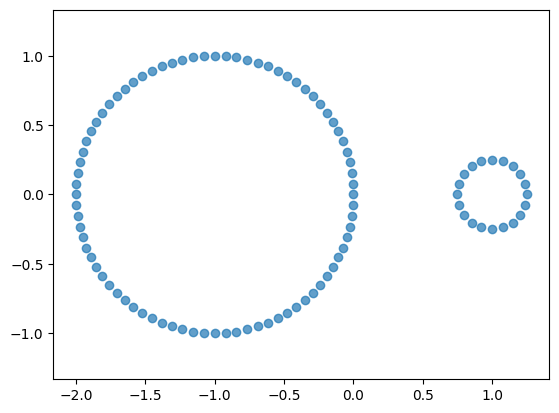}
        \caption{Point cloud.}
    \end{subfigure}
    ~
    \begin{subfigure}[t]{0.3\textwidth}
        \centering
        \includegraphics[height=1.5in]{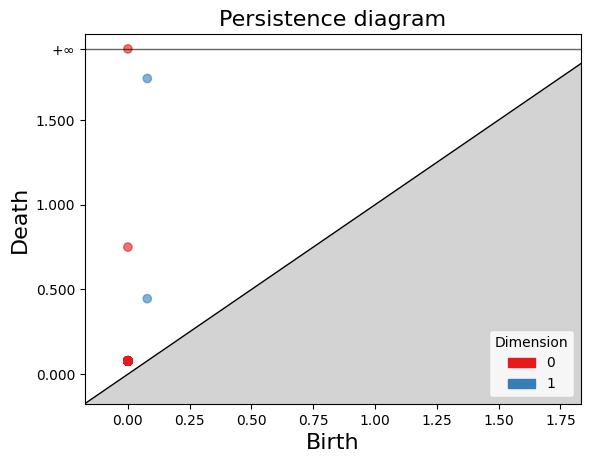}
        \caption{Persistence diagram.}
    \end{subfigure}
    ~
    \begin{subfigure}[t]{0.35\textwidth}
        \centering
        \includegraphics[height=1.5in]{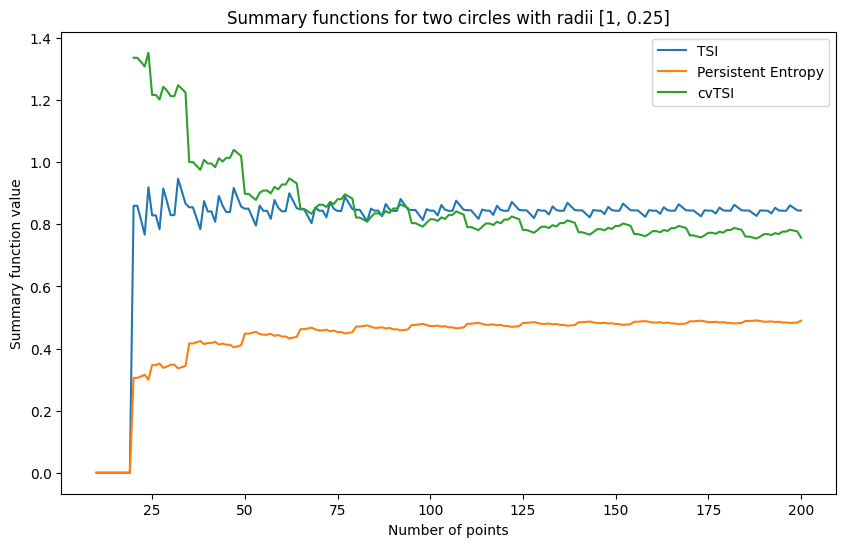}
        \caption{Statistics vs.\ sample size.}
    \end{subfigure}
    \caption{Two disjoint circles with radii $1$ and $\tfrac{1}{4}$. 
(A) Point cloud sampled with increasing density. 
(B) Persistence diagram showing two dominant $H_1$ features corresponding to the two loops. 
(C) Evolution of TSI, persistent entropy, and $cv\tsi$ as a function of sample size. 
}
    \label{fig:disjoint_circles_total}
\end{figure}


\bigbreak 

\paragraph{{\bf Intertwined circles:}}
We next consider two intersecting circles, as illustrated in Figure~\ref{fig:intertwined_circles_total}. In this configuration, the persistence diagram exhibits three prominent $1$-dimensional features corresponding to the two individual loops and the additional cycle created by their intersection. Unlike the disjoint case, the persistence lifetimes are no longer proportional, resulting in a more heterogeneous distribution of bar lengths. Furthermore, in this setting, the number of features changes depending on the amount of points due to small cycles appearing at the intersection points of the circles.

We observe that the TSI (again) stabilizes rapidly with increasing sample size, confirming its robustness to sampling density. However, small fluctuations remain due to the presence of short-lived bars, illustrating the sensitivity of TSI to local perturbations (recall Lemma~\ref{lemma:1_changed_bar}). Persistent entropy, on the other hand, reflects the overall balance of the normalized distribution and is less sensitive to the relative separation between dominant features. This highlights a key distinction: while entropy measures distributional uniformity, the TSI emphasizes dispersion and relative scale differences among persistence lifetimes. Finally, $cv\tsi$ exhibits pronounced oscillations due to its sensitivity to short-lived features, illustrating the impact of local perturbations on scale-invariant summaries.

\begin{figure}[H]
    \centering
    \begin{subfigure}[t]{0.3\textwidth}
        \centering
        \includegraphics[height=1.5in]{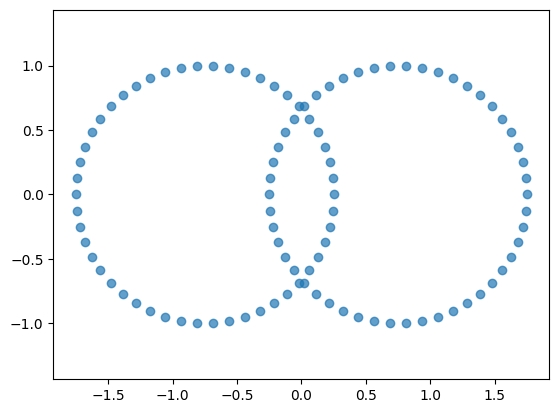}
        \caption{Point cloud.}
    \end{subfigure}
    ~
    \begin{subfigure}[t]{0.3\textwidth}
        \centering
        \includegraphics[height=1.5in]{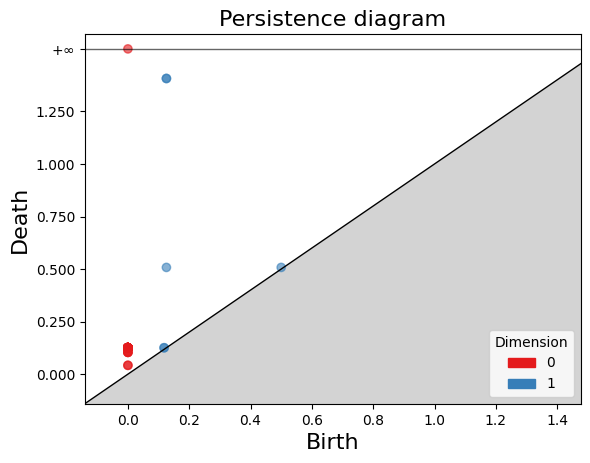}
        \caption{Persistence diagram.}
    \end{subfigure}
    ~
    \begin{subfigure}[t]{0.35\textwidth}
        \centering
        \includegraphics[height=1.5in]{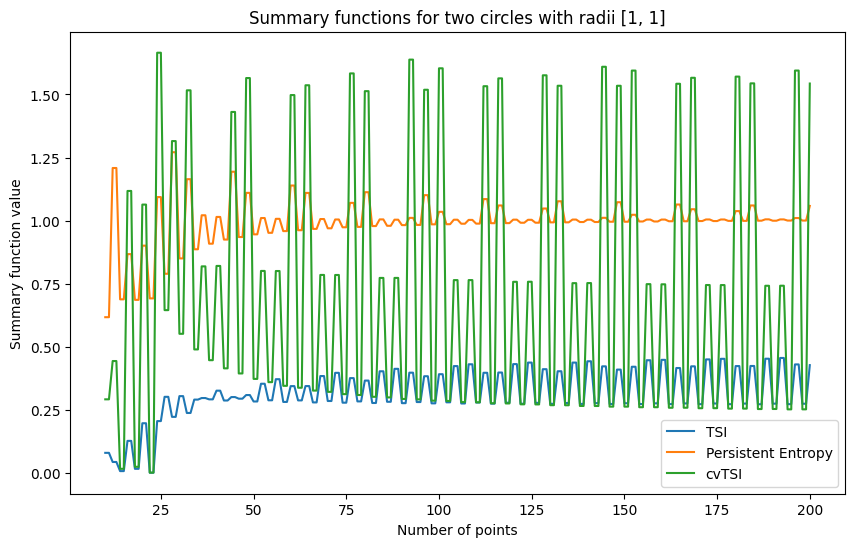}
        \caption{Statistics vs.\ sample size.}
    \end{subfigure}
    \caption{
Two intertwined circles of equal radius. 
(A) Point cloud. 
(B) Persistence diagram showing multiple $H_1$ features arising from the overlapping structure. 
(C) Evolution of TSI, persistent entropy, and $cv\tsi$ as a function of sample size.
}
    \label{fig:intertwined_circles_total}
\end{figure}

\bigbreak 

\paragraph{\bf Random sampling effects}

We now introduce randomness by sampling points independently and uniformly along two circles, as illustrated in Figure~\ref{fig:sampled_circles_total}. The resulting persistence diagrams vary across realizations, and we therefore estimate the statistics using Monte Carlo simulations. At low sampling densities, the underlying topological structure is only partially resolved, leading to incomplete and unstable persistence diagrams. As the number of sampled points increases, the dominant features become more clearly defined, resulting in a gradual increase and eventual stabilization of both the TSI and the persistent entropy.

A key distinction emerges in the variability of the estimators under stochastic sampling. While the TSI exhibits an initial growth phase reflecting the progressive resolution of topological features, it stabilizes relatively quickly across realizations once the dominant structure is captured. Persistent entropy exhibits smoother convergence due to its normalization. In contrast, $cv\tsi$ shows significantly higher variability, as small fluctuations in short-lived bars are amplified by normalization with respect to the mean lifetime.

\begin{figure}[H]
    \centering
    \begin{subfigure}[t]{0.3\textwidth}
        \centering
        \includegraphics[height=1.5in]{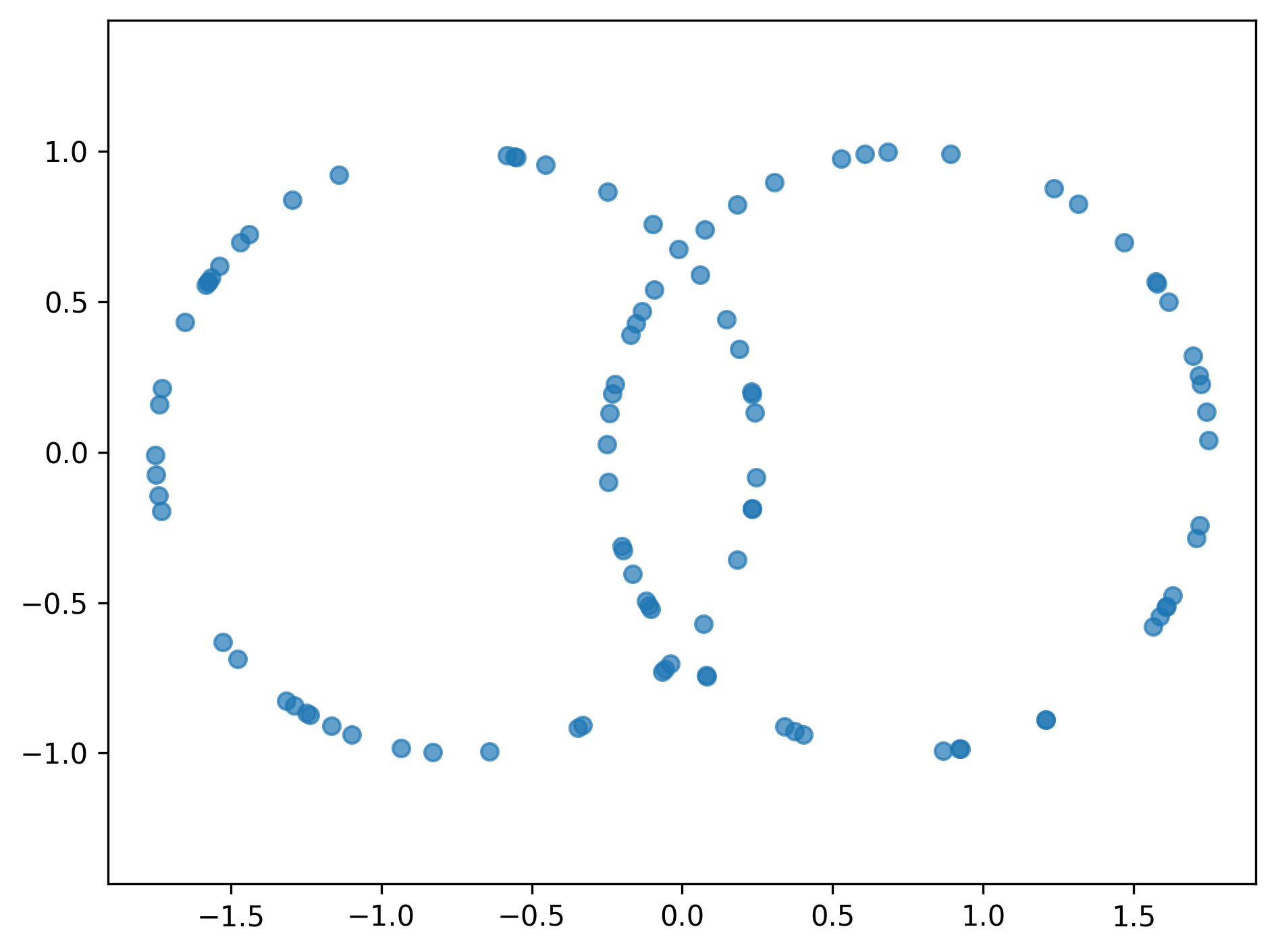}
        \caption{Point cloud. \label{subfig:sampled_circles}}
    \end{subfigure}
    ~
    \begin{subfigure}[t]{0.3\textwidth}
        \centering
        \includegraphics[height=1.5in]{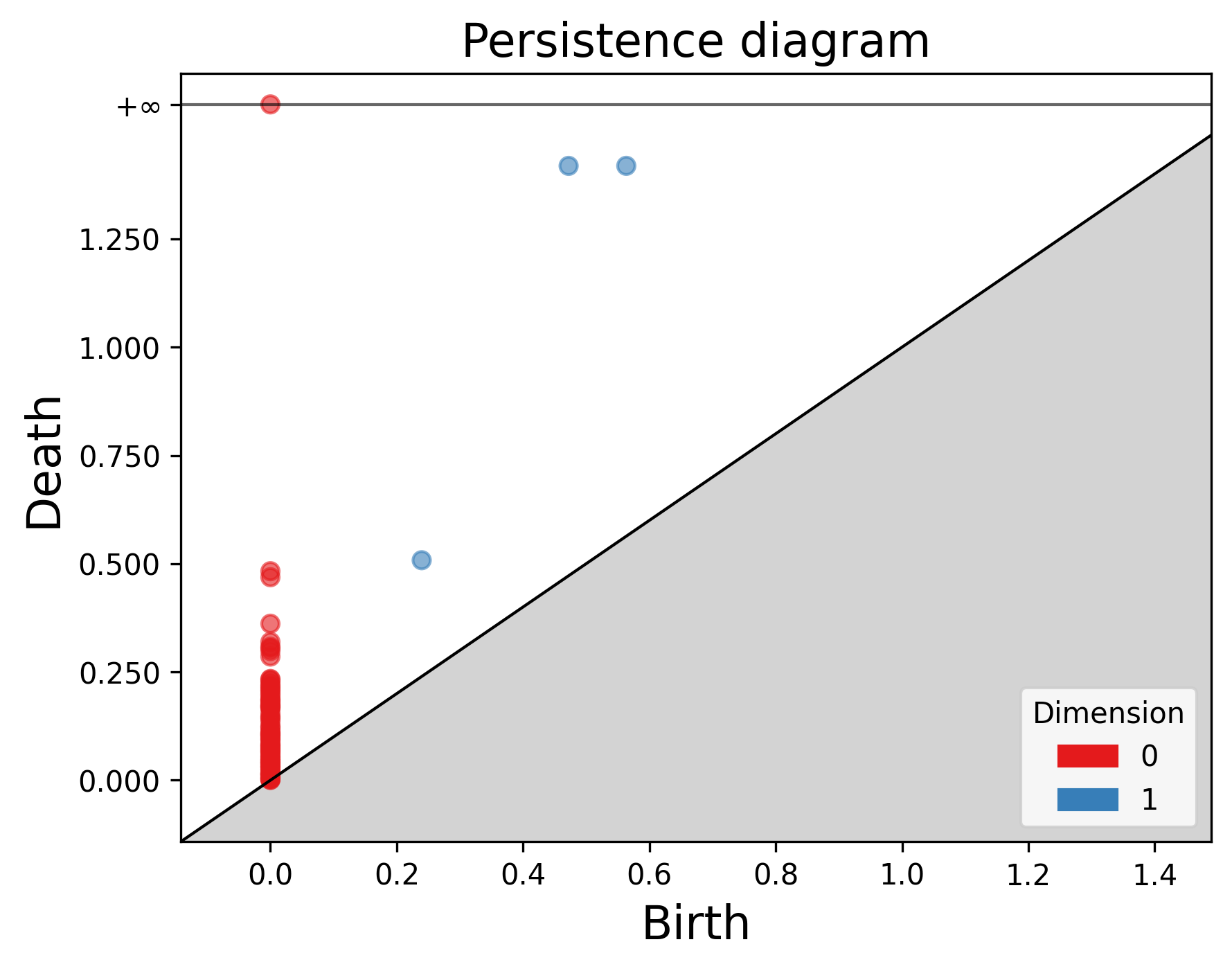}
        \caption{Persistence diagram.}
    \end{subfigure}
    ~
    \begin{subfigure}[t]{0.35\textwidth}
        \centering
        \includegraphics[height=1.5in]{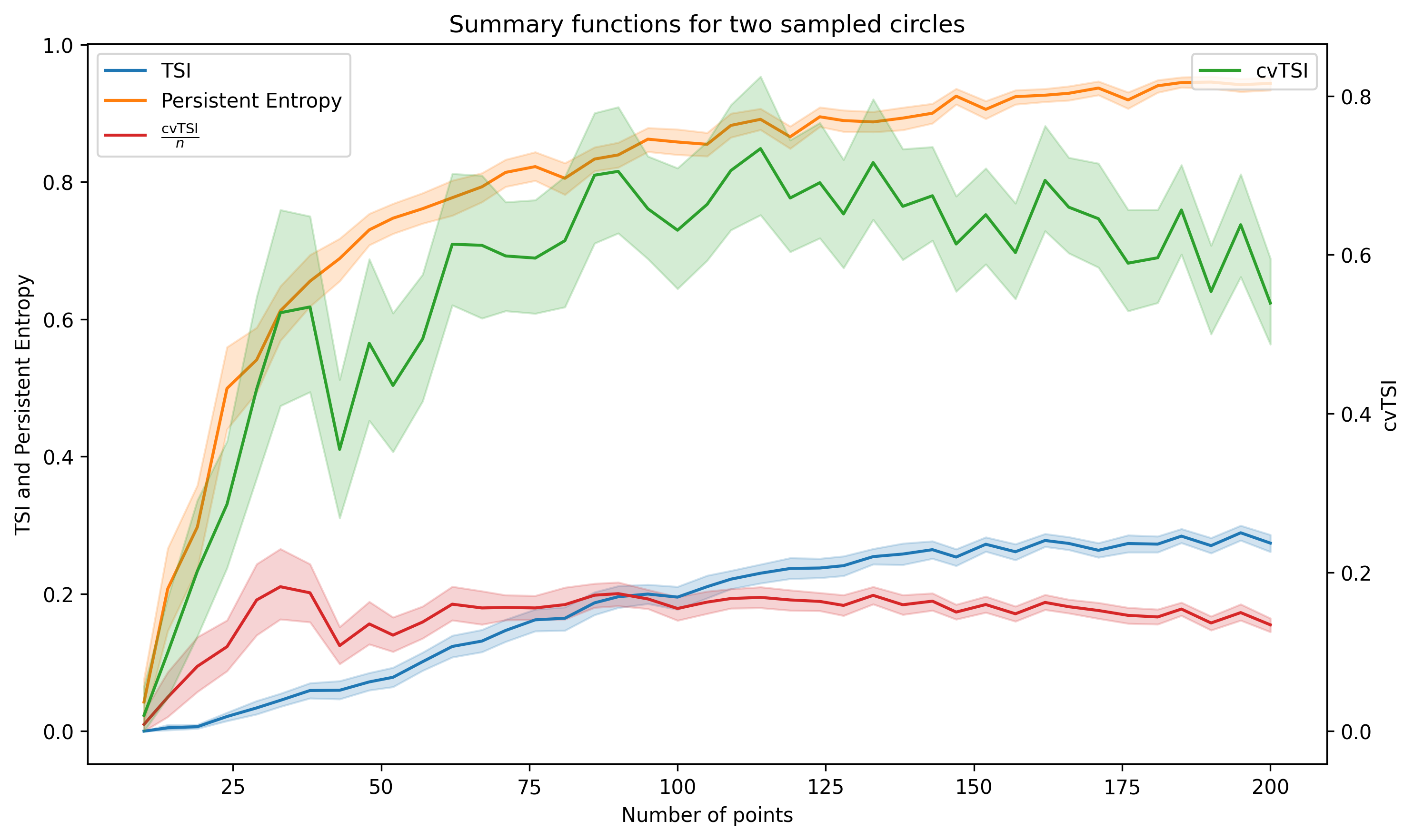}
        \caption{Monte Carlo estimates.}
    \end{subfigure}
    \caption{
Random sampling on intertwined circles. 
(A) A random realization of the sampled point cloud. 
(B) Corresponding persistence diagram, illustrating variability due to stochastic sampling. 
(C) Monte Carlo estimates (mean and variability) of TSI, persistent entropy, $cv\tsi$ and $\frac{cv\tsi}{n}$ as functions of sample size based on 100 simulations per point.
}
    \label{fig:sampled_circles_total}
\end{figure}

To conclude, this experiment highlights a fundamental distinction between the summaries. As noted before, persistent entropy provides a stable, smooth characterization of the normalized lifetime distribution, while the TSI captures the progressive resolution of geometric structure as sampling density increases. In contrast, $cv\tsi$ amplifies fluctuations arising from short-lived features, making it highly sensitive to local perturbations. These observations suggest that, under stochastic sampling, the TSI offers a robust descriptor of structural heterogeneity, whereas entropy emphasizes distributional balance and $cv\tsi$ serves as a fine-scale fluctuation detector.

\subsubsection{Alpha complex} 
All preceding experiments were conducted using the Vietoris-Rips complex. However, the Alpha complex serves as a common alternative due to its superior computational efficiency for filtrations on large point clouds. There are two primary distinctions between these complexes in this context. First, the Alpha complex is restricted to simplices up to dimension $1$, whereas the Rips complex permits simplices of arbitrarily high dimension. Second, the Alpha complex can contain $1$-dimensional cycles composed of three simplices, a configuration that is impossible in the Rips complex (where any triplet of pairwise connected vertices immediately forms a $2$-simplex). This latter distinction is particularly significant when analyzing $1$-dimensional features, as the Alpha complex introduces a multitude of short-lived topological features, or "noise", that can influence the resulting statistics.

Figure~\ref{fig:sampled_circles_alpha} illustrates the previously described experiment on random sampling effects, now conducted utilizing the Alpha complex. While the resulting persistence diagram retains the prominent features observed in the Rips diagram, it now also incorporates these additional short-lived features. For large sample sizes, the persistent entropy and TSI remain largely unperturbed by this noise; discrepancies only become apparent for small values of $n$. Conversely, the $cv\tsi$ operates on a vastly different scale, as it is highly sensitive to the total number of features, in accordance with Theorem~\ref{thm:cvtsi_bounds}. To meaningfully compare the two filtrations using this statistic, we examine the normalized quantity $\frac{cv\tsi}{n}$. This metric exhibits a substantial difference, yielding approximately double the value observed in the Rips filtration, thereby demonstrating its capacity to capture the structural changes induced by the choice of complex (compare with Figure~\ref{fig:sampled_circles_total}).

To conclude, we note that the use of the Alpha filtration introduces topological noise that can distort quantitative precision. Should computational constraints require the Alpha complex, the analysis appears to remain qualitatively sound; our results indicate that global persistent patterns are consistent across both methods.

\begin{figure}[H]
    \begin{subfigure}[t]{0.33\textwidth}
        \centering
        \includegraphics[height=2in]{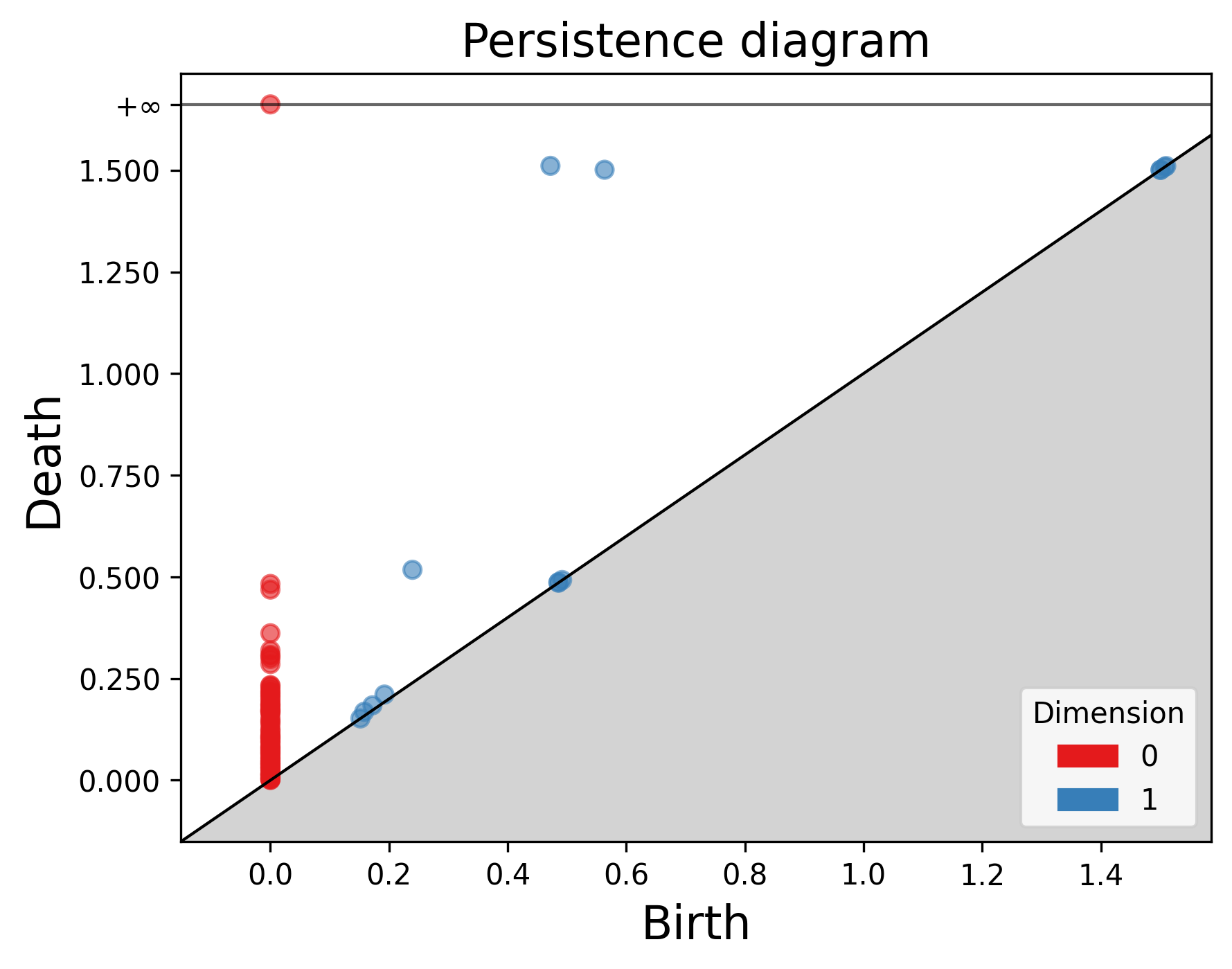}
        \caption{Persistence diagram.}
    \end{subfigure}
    ~
    \begin{subfigure}[t]{0.66\textwidth}
        \centering
        \includegraphics[height=2in]{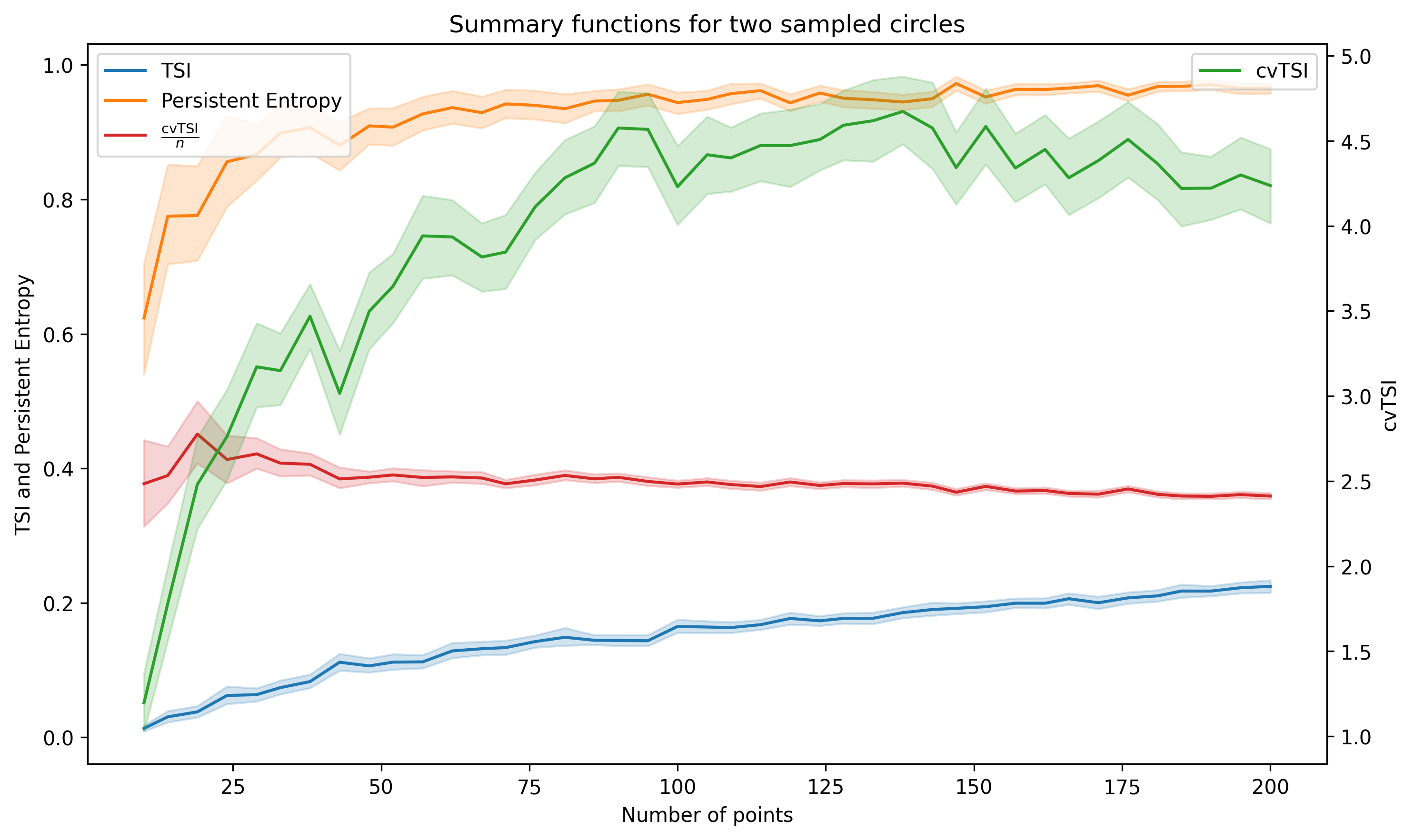}
        \caption{Monte Carlo estimates.}
    \end{subfigure}
    \caption{
Random sampling on intertwined circles using the Alpha complex. 
(A) Persistence diagram corresponding to the point cloud of Figure~\ref{fig:sampled_circles_total}(A), illustrating additional small features due to the Alpha complex. 
(B) Monte Carlo estimates (mean and variability) of TSI, persistent entropy, $cv\tsi$ and $\frac{cv\tsi}{n}$ as functions of sample size based on 100 simulations per point.
}
    \label{fig:sampled_circles_alpha}
\end{figure}


\subsection{Robustness under noise}

We now analyze the effect of noise on the persistence-based summaries (see Figure~\ref{fig:noisy_circles_total}). As noise increases, the persistence diagram becomes progressively dominated by short-lived features, reflecting the degradation of the underlying geometric structure.

We consider two distinct models of noise. The first is Gaussian noise, in which $200$ sampled points are perturbed by an error term $\varepsilon \sim \mathcal{N}(0, \sigma I)$, serving to mimic measurement inaccuracies. By varying the standard deviation $\sigma$, we can evaluate the impact of different noise intensities on the topological features. The second model is uniform noise. In this scenario, alongside $100$ points sampled from the underlying space, we introduce additional points sampled from a uniform distribution over the bounding rectangle $[-1.75, 1.75] \times [-1, 1]$. This distribution represents random outliers and faulty samples that are fundamentally irrelevant to the topological analysis. The number of these supplementary points is set to $100r$, where the parameter $r \in [0, 1]$ dictates the intensity of the noise and is systematically varied throughout the experiments. 


Under Gaussian noise (Figure~\ref{fig:noisy_circles_total}(A)), the TSI decreases rapidly toward zero, indicating a loss of variability in persistence lifetimes as dominant features are destroyed. In contrast, persistent entropy increases monotonically, reflecting a more uniform distribution of persistence mass across a multitude of short-lived features. The behavior of $cv\tsi$ is more nuanced. As noise increases, $cv\tsi$ initially rises, reaching a peak when the persistence diagram contains a mixture of prominent and short-lived features, and subsequently decreases as the diagram becomes dominated by uniformly small lifetimes. This non-monotone behavior highlights the sensitivity of $cv\tsi$ to intermediate regimes where both structure and noise coexist. By utilizing this metric in conjunction with the TSigI, we can infer several intrinsic properties of the point cloud. The TSigI is initially high due to a predominance of prominent features, and it decreases with added noise owing to the increasing number of short-lived features. Consequently, under low noise conditions, a high TSigI paired with a low $cv\tsi$ indicates minimal variance among the prominent features. At intermediate noise levels, a diminished TSigI coupled with a high $cv\tsi$ suggests a distinct split between prominent and short-lived features. Finally, at high noise levels, concurrently low values of both the TSigI and $cv\tsi$ denote a topological landscape composed almost entirely of short-lived features.

Under uniform noise (Figure~\ref{fig:noisy_circles_total}(B)), similar qualitative trends are observed, but the presence of outliers introduces additional spurious features. In this setting, the TSI again decreases, while persistent entropy increases more sharply due to the broader spread of persistence values and a corresponding increase in the total number of bars. The normalized index $cv\tsi$ stabilizes at an intermediate level, reflecting a balance between structured and unstructured topological contributions. The TSigI follows a trajectory similar to that observed under Gaussian noise, again signaling the presence of prominent features at low noise levels and short-lived features at higher intensities.

\begin{figure}[H]
    \centering
    \begin{subfigure}[t]{0.49\textwidth}
        \centering
        \includegraphics[height=2in]{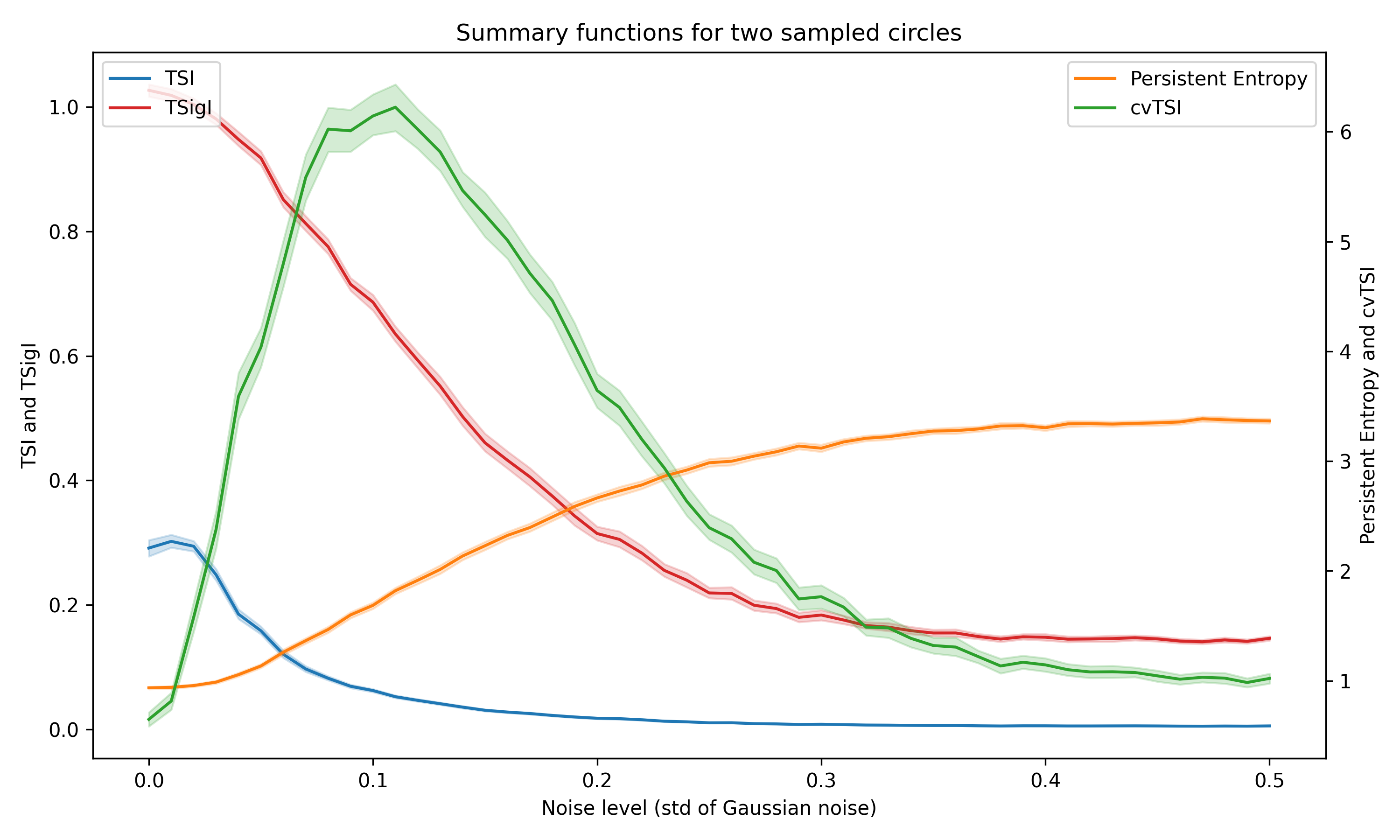}
        \caption{Gaussian noise.}
    \end{subfigure}
    ~
    \begin{subfigure}[t]{0.49\textwidth}
        \centering
        \includegraphics[height=2in]{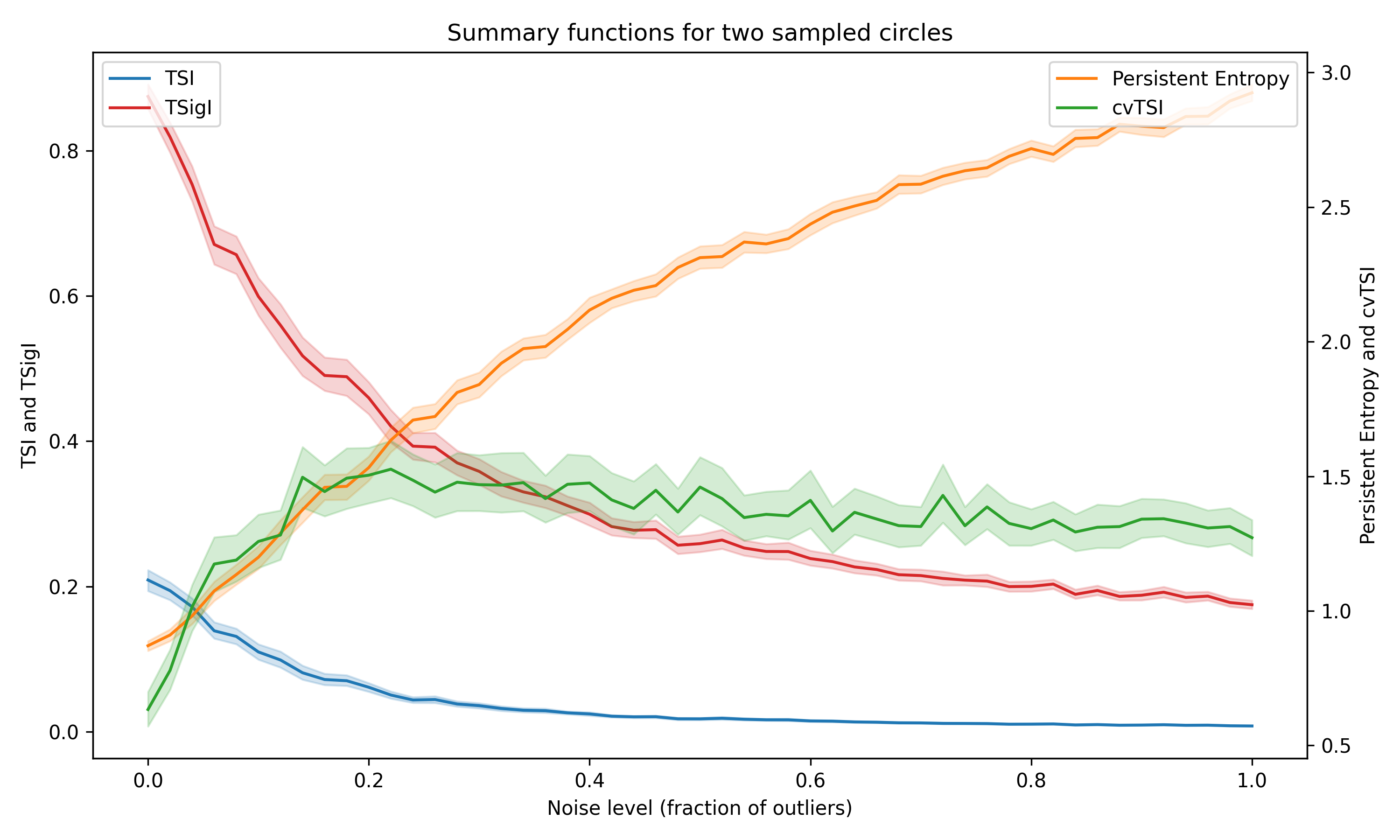}
        \caption{Uniform noise.}
    \end{subfigure}
    \caption{Effect of noise on TSI, TSigI, entropy, and $cv\tsi$.}
    \label{fig:noisy_circles_total}
\end{figure}


\subsection{Time series: geometric Brownian motion}

We now consider synthetic time series generated by geometric Brownian motion (GBM). Using Takens embedding (the interested readers is referred to \cite{schilling2014brownian, skraba_topological_2012}), we reconstruct a point cloud from each time series and compute the associated persistence diagrams (for an illustration see Figures~\ref{fig:gbm_total1} and \ref{fig:gbm_total2}). 

\begin{figure}[H]
    \centering
    \begin{subfigure}[t]{0.32\textwidth}
    \centering
        \includegraphics[height=1.5in]{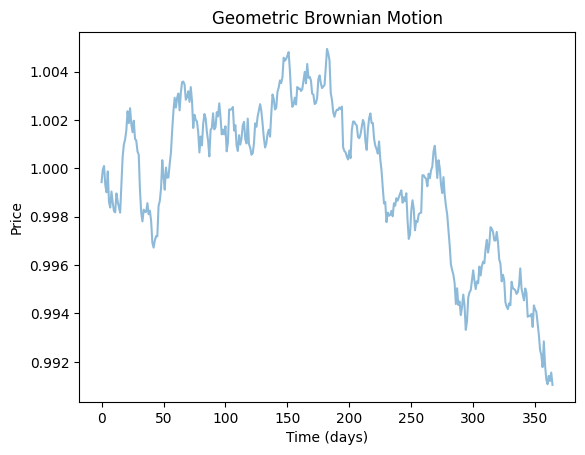}
        \caption{\label{fig:gbm_example} Realization of geometric Brownian motion.}
    \end{subfigure}
    ~
    \begin{subfigure}[t]{0.32\textwidth}
    \centering
        \includegraphics[height=1.5in]{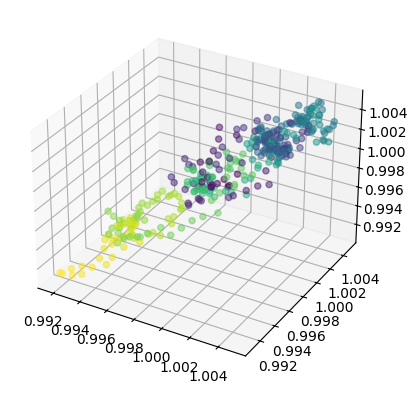}
        \caption{\label{fig:gbm_exampleTaken} Takens embedding.}
    \end{subfigure}
    ~
    \begin{subfigure}[t]{0.32\textwidth}
        \centering
        \includegraphics[height=1.5in]{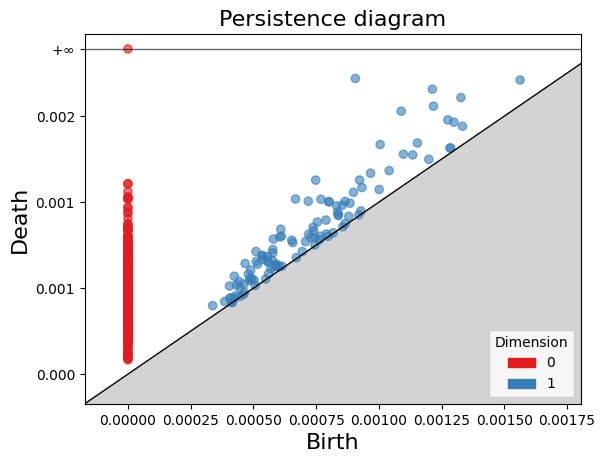}
        \caption{Persistence diagram.}
    \end{subfigure}
    \caption{(A) A simulated example of geometric Brownian motion with drift and volatility $(\mu,\sigma)=(0,0.01)$ (B) 3 dimensional Takens embedding of the curve with delay parameter $\tau=3$. (C) Corresponding persistence diagram of Alpha filtration.}
    \label{fig:gbm_total1}
\end{figure}

\begin{figure}[H]
\centering
    \begin{subfigure}[t]{0.42\textwidth}
        \centering
        \includegraphics[height=1.5in]{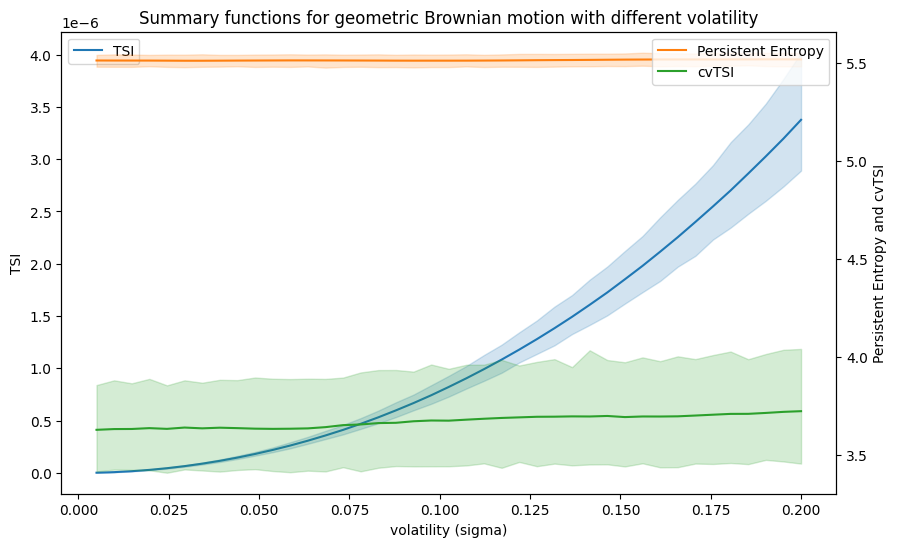}
        \caption{\label{fig:gbm_volatility} Different amount of volatility for zero drift.}
    \end{subfigure}
    ~
    \begin{subfigure}[t]{0.42\textwidth}
        \centering
        \includegraphics[height=1.5in]{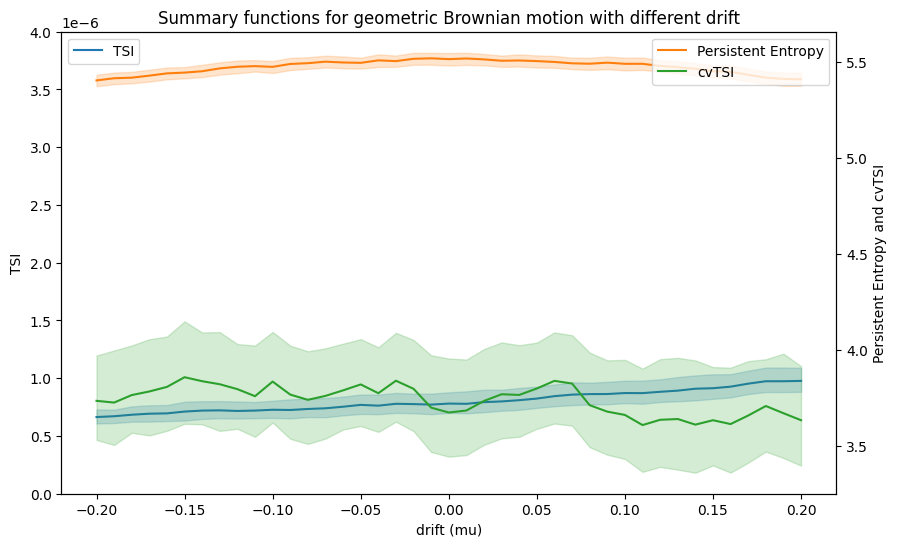}
        \caption{\label{fig:gbm_drift} Different amount of drift for 0.1 volatility.}
    \end{subfigure}
    \caption{TSI, persistent entropy, and $cv\tsi$ as functions of volatility (A) and drift (B). Each value is averaged over 100 Monte Carlo simulations.}
    \label{fig:gbm_total2}
\end{figure}

The persistence diagrams of GBM exhibit a large number of short-lived features, reflecting the absence of strong underlying topological structure. This is consistent with the stochastic and memoryless nature of the process.

To assess sensitivity to model parameters, we vary both the drift $\mu$ and the volatility $\sigma$. The results reveal a clear contrast between these two effects. When varying the drift, all three summary statistics remain largely stable across the parameter range. In particular, the TSI exhibits only minor variation, indicating that it is primarily driven by fluctuations rather than deterministic trends. Persistent entropy displays similarly weak, but significant dependence on $\mu$, while $cv\tsi$ shows higher variability but no systematic trend.

In contrast, increasing volatility leads to noticeable changes in all summary statistics. The TSI increases with volatility, reflecting the growing dispersion in persistence lifetimes induced by stronger stochastic fluctuations. Persistent entropy also increases, capturing the more uniform distribution of persistence mass across features, while $cv\tsi$ exhibits amplified variability due to its normalization.

These findings highlight a key distinction: persistence-based summaries are largely insensitive to drift in GBM, but respond strongly to volatility. In particular, the TSI captures the magnitude of stochastic fluctuations rather than deterministic components of the dynamics, reinforcing its interpretation as a measure of structural dispersion.


\section{Discussion and Future Work}\label{Sec:Discussion and Future Work}
The Topological Stability Index (TSI) provides a new scalar summary of persistence barcodes that complements existing entropy-based approaches. Together with the Topological Signal Index (TSigI), which captures the typical scale of persistence lifetimes, it forms a two-dimensional summary of persistence barcodes, encoding both dispersion and magnitude.

The normalized version of $\tsi$, $cv\tsi$, establishes a precise connection to information-theoretic quantities through its exact relation to Rényi entropy of order two, showing that variance-based and entropy-based summaries are closely linked but emphasize different aspects of the data. Furthermore, while persistent entropy captures the relative distribution of persistence lifetimes, the $cv\tsi$ measures their absolute dispersion, and TSigI captures their characteristic scale, making the pair particularly sensitive to dominant features and structural variability. 

Despite these advantages, several limitations should be acknowledged. Both $cv\tsi$ and TSigI are sensitive to changes in the number of bars, especially under the insertion of short-lived features, and therefore depend on the quality and stability of the underlying persistence diagram. Moreover, as low-dimensional summaries, they compress the full topological information into a small number of values and do not capture the evolution of features across the filtration.

These limitations point to several directions for future research. A natural extension is the development of functional analogues of the TSI and TSigI within the persistence-curve framework, enabling time-resolved analysis of both structural variability and signal strength. From a statistical perspective, it would be valuable to study the asymptotic behavior of these quantities under sampling and noise, as well as their use in hypothesis testing and inference. Further applications in time series analysis, clustering, and anomaly detection may also clarify their practical strengths and limitations.

\printbibliography

@book{schilling2014brownian,
  title={Brownian motion: an introduction to stochastic processes},
  author={Schilling, Ren{\'e} L and Partzsch, Lothar},
  year={2014},
  publisher={Walter de Gruyter GmbH \& Co KG}
}

@article{zomorodian_computing_2005,
	title = {Computing {Persistent} {Homology}},
	volume = {33},
	issn = {1432-0444},
	url = {https://doi.org/10.1007/s00454-004-1146-y},
	doi = {10.1007/s00454-004-1146-y},
	language = {en},
	number = {2},
	journal = {Discrete \& Computational Geometry},
	author = {Zomorodian, Afra and Carlsson, Gunnar},
	month = feb,
	year = {2005},
	pages = {249--274},
}

@article{chazal_robust_2018,
    title = {Robust {Topological} {Inference}: {Distance} {To} a {Measure} and {Kernel} {Distance}},
    volume = {18},
    issn = {1533-7928},
    shorttitle = {Robust {Topological} {Inference}},
    url = {http://jmlr.org/papers/v18/15-484.html},
    number = {159},
    journal = {Journal of Machine Learning Research},
    author = {Chazal, Frédéric and Fasy, Brittany and Lecci, Fabrizio and Michel, Bertrand and Rinaldo, Alessandro and Wasserman, Larry},
    year = {2018},
    pages = {1--40},
}

@article{robinson_hypothesis_2017,
    title = {Hypothesis testing for topological data analysis},
    volume = {1},
    issn = {2367-1734},
    url = {https://doi.org/10.1007/s41468-017-0008-7},
    doi = {10.1007/s41468-017-0008-7},
    language = {en},
    number = {2},
    journal = {Journal of Applied and Computational Topology},
    author = {Robinson, Andrew and Turner, Katharine},
    month = dec,
    year = {2017},
    pages = {241--261},
}

@article{atienza_stability_2020,
    title = {On the stability of persistent entropy and new summary functions for topological data analysis},
    volume = {107},
    issn = {0031-3203},
    url = {https://www.sciencedirect.com/science/article/pii/S0031320320303125},
    doi = {10.1016/j.patcog.2020.107509},
    journal = {Pattern Recognition},
    author = {Atienza, Nieves and Gonzalez-Díaz, Rocio and Soriano-Trigueros, Manuel},
    month = nov,
    year = {2020},
    pages = {107509},
}

@article{atienza_persistent_2019,
    title = {Persistent entropy for separating topological features from noise in vietoris-rips complexes},
    volume = {52},
    issn = {1573-7675},
    url = {https://doi.org/10.1007/s10844-017-0473-4},
    doi = {10.1007/s10844-017-0473-4},
    language = {en},
    number = {3},
    journal = {Journal of Intelligent Information Systems},
    author = {Atienza, Nieves and Gonzalez-Diaz, Rocio and Rucco, Matteo},
    month = jun,
    year = {2019},
    pages = {637--655},
}

@article{chung_persistence_2022,
    title = {Persistence {Curves}: {A} canonical framework for summarizing persistence diagrams},
    volume = {48},
    issn = {1572-9044},
    shorttitle = {Persistence {Curves}},
    url = {https://doi.org/10.1007/s10444-021-09893-4},
    doi = {10.1007/s10444-021-09893-4},
    language = {en},
    number = {1},
    journal = {Advances in Computational Mathematics},
    author = {Chung, Yu-Min and Lawson, Austin},
    month = jan,
    year = {2022},
    pages = {6},
}

@article{bhatia_better_2000,
    title = {A {Better} {Bound} on the {Variance}},
    volume = {107},
    issn = {0002-9890},
    url = {https://www.jstor.org/stable/2589180},
    doi = {10.2307/2589180},
    number = {4},
    journal = {The American Mathematical Monthly},
    publisher = {[Taylor \& Francis, Ltd., Mathematical Association of America]},
    author = {Bhatia, Rajendra and Davis, Chandler},
    year = {2000},
    pages = {353--357},
}

@inproceedings{rucco_characterisation_2016,
    address = {Cham},
    title = {Characterisation of the {Idiotypic} {Immune} {Network} {Through} {Persistent} {Entropy}},
    isbn = {978-3-319-29228-1},
    doi = {10.1007/978-3-319-29228-1_11},
    language = {en},
    booktitle = {Proceedings of {ECCS} 2014},
    publisher = {Springer International Publishing},
    author = {Rucco, Matteo and Castiglione, Filippo and Merelli, Emanuela and Pettini, Marco},
    editor = {Battiston, Stefano and De Pellegrini, Francesco and Caldarelli, Guido and Merelli, Emanuela},
    year = {2016},
    pages = {117--128},
}

@article{fasy_confidence_2014,
    title = {Confidence sets for persistence diagrams},
    volume = {42},
    issn = {0090-5364, 2168-8966},
    url = {https://projecteuclid.org/journals/annals-of-statistics/volume-42/issue-6/Confidence-sets-for-persistence-diagrams/10.1214/14-AOS1252.full},
    doi = {10.1214/14-AOS1252},
    number = {6},
    journal = {The Annals of Statistics},
    publisher = {Institute of Mathematical Statistics},
    author = {Fasy, Brittany Terese and Lecci, Fabrizio and Rinaldo, Alessandro and Wasserman, Larry and Balakrishnan, Sivaraman and Singh, Aarti},
    month = dec,
    year = {2014},
    pages = {2301--2339},
}

@article{hiraoka_limit_2018,
    title = {Limit theorems for persistence diagrams},
    volume = {28},
    issn = {1050-5164, 2168-8737},
    url = {https://projecteuclid.org/journals/annals-of-applied-probability/volume-28/issue-5/Limit-theorems-for-persistence-diagrams/10.1214/17-AAP1371.full},
    doi = {10.1214/17-AAP1371},
    number = {5},
    journal = {The Annals of Applied Probability},
    publisher = {Institute of Mathematical Statistics},
    author = {Hiraoka, Yasuaki and Shirai, Tomoyuki and Trinh, Khanh Duy},
    month = oct,
    year = {2018},
    pages = {2740--2780},
}

@article{mileyko_probability_2011,
    title = {Probability measures on the space of persistence diagrams},
    volume = {27},
    issn = {0266-5611},
    url = {https://doi.org/10.1088/0266-5611/27/12/124007},
    doi = {10.1088/0266-5611/27/12/124007},
    language = {en},
    number = {12},
    journal = {Inverse Problems},
    author = {Mileyko, Yuriy and Mukherjee, Sayan and Harer, John},
    month = nov,
    year = {2011},
    pages = {124007},
}

@article{ali_survey_2023,
    title = {A {Survey} of {Vectorization} {Methods} in {Topological} {Data} {Analysis}},
    volume = {45},
    issn = {1939-3539},
    url = {https://ieeexplore.ieee.org/abstract/document/10235748},
    doi = {10.1109/TPAMI.2023.3308391},
    number = {12},
    journal = {IEEE Transactions on Pattern Analysis and Machine Intelligence},
    author = {Ali, Dashti and Asaad, Aras and Jimenez, Maria-Jose and Nanda, Vidit and Paluzo-Hidalgo, Eduardo and Soriano-Trigueros, Manuel},
    month = dec,
    year = {2023},
    pages = {14069--14080},
}

@article{bubenik_statistical_2015,
    title = {Statistical {Topological} {Data} {Analysis} using {Persistence} {Landscapes}},
    volume = {16},
    issn = {1533-7928},
    url = {http://jmlr.org/papers/v16/bubenik15a.html},
    number = {3},
    journal = {Journal of Machine Learning Research},
    author = {Bubenik, Peter},
    year = {2015},
    pages = {77--102},
}

@inproceedings{anai_dtm-based_2020,
    address = {Cham},
    title = {{DTM}-{Based} {Filtrations}},
    isbn = {978-3-030-43408-3},
    doi = {10.1007/978-3-030-43408-3_2},
    language = {en},
    booktitle = {Topological {Data} {Analysis}},
    publisher = {Springer International Publishing},
    author = {Anai, Hirokazu and Chazal, Frédéric and Glisse, Marc and Ike, Yuichi and Inakoshi, Hiroya and Tinarrage, Raphaël and Umeda, Yuhei},
    editor = {Baas, Nils A. and Carlsson, Gunnar E. and Quick, Gereon and Szymik, Markus and Thaule, Marius},
    year = {2020},
    pages = {33--66},
}

@incollection{renyi1961,
  author    = {Alfr{\'e}d R{\'e}nyi},
  title     = {On measures of entropy and information},
  booktitle = {Proceedings of the Fourth Berkeley Symposium on Mathematical Statistics and Probability, Volume 1: Contributions to the Theory of Statistics},
  publisher = {University of California Press},
  address   = {Berkeley, CA},
  year      = {1961},
  pages     = {547--561}
}

@article{diamantis_shape_2025,
    title = {The Shape of Data: Topology Meets Analytics
A Practical Introduction to Topological Analytics and the Stability Index (TSI) in Business},
    language = {en},
    author = {Diamantis, Ioannis},
    year = {2025},
    eprint = {2511.13503},
    archivePrefix = {arXiv},
    primaryClass = {stat.ML},
    note = {Preprint},
}

@article{Mileyko_2011,
doi = {10.1088/0266-5611/27/12/124007},
url = {https://doi.org/10.1088/0266-5611/27/12/124007},
year = {2011},
month = {},
publisher = {},
volume = {27},
number = {12},
pages = {124007},
author = {Mileyko, Yuriy and Mukherjee, Sayan and Harer, John},
title = {Probability measures on the space of persistence diagrams},
journal = {Inverse Problems}
}

@article{cohen-steiner_stability_2007,
    title = {Stability of Persistence Diagrams},
    volume = {37},
    issn = {0179-5376},
    url = {https://doi.org/10.1007/s00454-006-1276-5},
    doi = {10.1007/s00454-006-1276-5},
    journal = {Discrete \& Computational Geometry},
    author = {Cohen-Steiner, David and Edelsbrunner, Herbert and Harer, John},
    year = {2007},
    number = {1},
    pages = {103--120},
}

@inproceedings{skraba_topological_2012,
	title = {Topological {Analysis} of {Recurrent} {Systems}},
	url = {https://urn.kb.se/resolve?urn=urn:nbn:se:kth:diva-107210},
	language = {eng},
	author = {Skraba, Primoz and de Silva, Vin and Vejdemo-Johansson, Mikael},
	year = {2012},
	pages = {1--5},
    booktitle = {NIPS 2012 Workshop on Algebraic Topology and Machine Learning}
}

@article{LEHMER1971183,
title = {On the compounding of certain means},
journal = {Journal of Mathematical Analysis and Applications},
volume = {36},
number = {1},
pages = {183-200},
year = {1971},
issn = {0022-247X},
doi = {https://doi.org/10.1016/0022-247X(71)90029-1},
url = {https://www.sciencedirect.com/science/article/pii/0022247X71900291},
author = {D.H Lehmer}
}
\end{document}